\documentclass[16pt]{article}
\usepackage{amsmath,amsfonts,amsthm,amssymb}
\usepackage{amssymb,amsmath,graphics}

 \usepackage{epsfig,psfrag}
\usepackage{graphicx}
\usepackage{psfig}
 \usepackage{verbatim}

\newtheorem{theorem}{Theorem}[section]
\newtheorem{lemma}[theorem]{Lemma}
\newtheorem{proposition}[theorem]{Proposition}

\newcommand\RR{{\Bbb R}}

\newcommand\ZZ{{\Bbb Z}}
\newcommand\TT{{\Bbb T}}

\newcommand\diam{{\rm diam}}
\newcommand\supp{{\rm supp}}
\begin{document}
\title{Nearly Tight
Frames and Space-Frequency Analysis\\ on Compact Manifolds
}
\author{Daryl Geller\\
\footnotesize\texttt{{Department of Mathematics, Stony Brook University, Stony Brook, NY 11794-3651}}\\
\footnotesize\texttt{{daryl@math.sunysb.edu}}\\
 \thanks{This work was partially  supported by the Marie Curie Excellence Team Grant MEXT-CT-2004-013477, Acronym MAMEBIA.}
Azita Mayeli \\
\footnotesize\texttt{{Department of Mathematics, Stony Brook University, Stony Brook, NY 11794-3651}}\\
\footnotesize\texttt{{amayeli@math.sunysb.edu}}}

\maketitle

\begin{abstract}
Let $\bf M$ be a smooth compact oriented Riemannian manifold of dimension $n$ without boundary,
and let $\Delta$ be the Laplace-Beltrami operator on ${\bf M}$. Say $0 \neq f \in
\mathcal{S}(\RR^+)$, and that $f(0) = 0$.  For $t > 0$, let $K_t(x,y)$ denote
the kernel of $f(t^2 \Delta)$.  Suppose $f$ satisfies Daubechies' criterion, and $b > 0$.
For each $j$, write ${\bf M}$ as a disjoint union of measurable sets $E_{j,k}$ with diameter
at most $ba^j$, and measure  comparable to $(ba^j)^n$ if $ba^j$ is sufficiently small.
Take $x_{j,k} \in E_{j,k}$.  We then show that the functions
$\phi_{j,k}(x)=\mu(E_{j,k})^{1/2} \overline{K_{a^j}}(x_{j,k},x)$ form a frame for
$(I-P)L^2({\bf M})$, for $b$ sufficiently small (here $P$ is the projection onto the constant functions).
Moreover, we show that the ratio of the frame bounds approaches 1
nearly quadratically as the dilation parameter approaches 1, so that the frame quickly becomes nearly tight (for
$b$ sufficiently small).
Moreover, based upon how well-localized a function
$F \in (I-P)L^2$ is in space and in frequency, we can describe which terms in
the summation $F \sim SF = \sum_j \sum_k \langle  F,\phi_{j,k}  \rangle  \phi_{j,k}$
are so small that they can be neglected.  If $n=2$ and $\bf M$ is the torus or the sphere,
and $f(s)=se^{-s}$ (the ``Mexican hat'' situation), we obtain two explicit approximate formulas
for the $\phi_{j,k}$, one to be used when $t$ is large, and one to be used when $t$ is small.\\

\footnotesize{
\begin{tabular}{lrl}
 Keywords and phrases:  &  \multicolumn{2}{l} {\em Frames, Wavelets, Continuous Wavelets, Spectral Theory, Schwartz Functions,}\\
 &  \hspace{-.1cm}  {\em  Time-Frequency  Analysis,  Manifolds, Sphere, Torus,  Pseudodifferential Operators.}\\
 \end{tabular}
 \vspace{.3cm}

 \begin{tabular}{lrl}
  &&  \hspace{-1cm} AMS  Classification;  {42C40, 42B20, 58J40, 58J35, 35P05.}
\end{tabular}
}
 \end{abstract}
{\bf Table of Contents}
\begin{itemize}
\item
\textit{Section 1: Introduction }
\item
\textit{Section 2: Frames}
\item
\textit{Section 3: Needlets and Mexican Needlets on the sphere}
\end{itemize}
\section{\large{Introduction}}
\label{sec:Introduction }

Say $f_0 \in {\cal S}(\RR^+)$ (the space of restrictions to $\RR^+$ of functions in ${\mathcal S}(\RR)$).
Say $f_0 \not\equiv 0$, and let
\[f(s) = sf_0(s).\]
One then has
the {\em Calder\'on formula}: if $c \in (0,\infty)$ is defined by
\[c = \int_0^{\infty} |f(t)|^2 \frac{dt}{t} = \int_0^{\infty} t|f_0(t)|^2 dt,\]
then for all $s > 0$,

\begin{equation}
\label{cald}
\int_0^{\infty} |f(ts)|^2 \frac{dt}{t} = c < \infty.
\end{equation}

The even function $h(\xi) = f(\xi^2)$ is then in ${\mathcal S}(\RR)$ and satisfies

\begin{equation}
\label{cald1}
\int_0^{\infty} |h(t\xi)|^2 \frac{dt}{t} = \frac{c}{2} < \infty.
\end{equation}

(In fact, all even functions in
${\mathcal S}(\RR)$ satisfying (\ref{cald1}) arise in this manner).

Its inverse Fourier transform $\check{h}$
is admissible (i.e. is a continuous wavelet).
 We prefer to write, formally,
\[ \check{h} = f(-d^2/dx^2) \delta; \]
the formal justfication being that

\[ (f(-d^2/dx^2) \delta)\hat{\:} = f(\xi^2) = h(\xi). \]
Thus $f(-d^2/dx^2) \delta$ is a continuous wavelet on $\RR$.

Discretizing (\ref{cald}), if $a \neq 1$ is sufficiently close to $1$,
one obtains a special form of
{\em Daubechies' condition}:
for all $s > 0$,
\begin{equation}
\label{daub}
0 < A_a \leq \sum_{j=-\infty}^{\infty} |f(a^{2j} s)|^2 \leq B_a < \infty,
\end{equation}
where
\begin{align}
\label{daubest}
A_a =& \frac{c}{2|\log a|} \left(1 - O(|(a-1)^2 (\log|a-1|)|\right), \\\label{daubest2} B_a=&
\frac{c}{2|\log a|}\left(1 + O(|(a-1)^2 (\log|a-1|)|)\right).
\end{align}
((\ref{daubest})  and (\ref{daubest2})  were proved in \cite{gm1}, Lemma
7.6, if $a > 1$.  If $a < 1$, replace $a$ by $1/a$ in that lemma.)
In particular, $B_a/A_a$ converges nearly quadratically to $1$ as $a \rightarrow 1$.  For example,
Daubechies calculated that if $f(s) = se^{-s}$ and $a=2^{1/3}$, then $B_a/A_a = 1.0000$  to four
significant digits.

Calder\'on's formula and Daubechies' condition are very important in the construction
of continuous wavelets and frames on $\RR$.

Our program is to construct (nearly tight) frames, and analogues of continuous wavelets, on
much more general spaces, by replacing the positive number $s$ in (\ref{cald}) and
(\ref{daub}) by a positive self-adjoint operator $T$ on a Hilbert space ${\cal H}$.
If $P$ is the projection onto the null space of $T$, by the spectral theorem
we obtain the relations

\begin{equation}
\label{gelmay1}
\int_0^{\infty} |f|^2(tT) \frac{dt}{t} = c(I-P)
\end{equation}

and

\begin{equation}
\label{gelmay2}
A_a(I-P) \leq \sum_{j=-\infty}^{\infty} |f|^2(a^{2j} T) \leq
B_a(I-P).
\end{equation}

(The integral in (\ref{gelmay1}) and the sum in (\ref{gelmay2}) converge strongly.
In (\ref{gelmay2}), $\sum_{j=-\infty}^{\infty} :=
\lim_{M, N \rightarrow \infty} \sum_{j=-M}^N$, taken in the strong operator
topology.)
(\ref{gelmay1}) and (\ref{gelmay2}) were justified in Section 2 of our earlier article
\cite{gmcw}.  (This article should be regarded as a sequel of \cite{gmcw}, although, if one is
willing to take some facts (including (\ref{gelmay1}) and (\ref{gelmay2})) for granted, it can be
read independently of \cite{gmcw}.)

Taking $T$ to be $-d^2/dx^2$ on $\RR$ leads to the continuous wavelet $f(-d^2/dx^2) \delta$ on
$\RR$. (Of course, on $\RR$, $P = 0$.)

We began our program of looking at more general positive self-adjoint operators $T$,
in our article \cite{gm1}.  There we took $T$ to be the sublaplacian $L$ on $L^2(G)$,
where $G$ is a stratifed Lie
group, and thereby obtained continuous wavelets and frames on such $G$.

In this article we will look at the (much more practical!)
situation in which $T$ is the Laplace-Beltrami operator on $L^2({\bf M})$,
where $({\bf M},g)$ is a smooth compact oriented Riemannian
manifold without boundary, of dimension $n$.  We will construct nearly tight frames
in this context.  Now $P$ will be the projection onto the
one-dimensional space of constant functions.
We constructed continuous wavelets on ${\bf M}$ in \cite{gmcw}.

We discussed the history of continuous wavelets and frames on manifolds at the end
of the introduction of \cite{gmcw}, and for space considerations, will not repeat it here.
In the last section of this article, we will present a detailed comparison of our
methods with the methods of Narcowich, Petrushev and Ward \cite{narc1}, \cite{narc2},
who worked on the sphere; we will explain the similarities and differences between their
methods and ours.\\

To see how frames can be obtained from (\ref{gelmay2}), suppose that, for any $t > 0$,
$K_t$ is the Schwartz kernel of $f(t^2T)$.  Thus, if $F \in L^2({\bf M})$,

\begin{equation}
\label{schkerf}
[f(t^2T)F](x) = \int_{\bf M} F(y) K_t(x,y) d\mu(y),
\end{equation}
here $\mu$ is the measure on ${\bf M}$ arising from integration with respect to the
volume form on ${\bf M}$.
Say now that $\int_{\bf M} F = 0$, so that $F = (I-P)F$.
By (\ref{gelmay2}),
\begin{equation}
\label{aasumba0}
A_a \langle F,F \rangle  \leq  \langle \sum_j |f|^2(a^{2j} T)F, F \rangle  \leq B_a \langle F,F \rangle .
\end{equation}
Thus
\begin{equation}
\label{aasumba}
A_a \langle F,F \rangle  \leq \sum_j  \langle f(a^{2j} T)F, f(a^{2j} T)F \rangle  \leq B_a \langle F,F \rangle ,
\end{equation}
so that
\begin{equation}
\label{kerfaj}
A_a \langle F,F \rangle  \leq \sum_j \int\left|\int K_{a^j}(x,y) F(y)d\mu(y)\right|^2 d\mu(x) \leq B_a \langle F,F \rangle
\end{equation}
Now, pick $b > 0$, and for each $j$, write ${\bf M}$
as a disjoint union of measurable sets $E_{j,k}$ with diameter at most $ba^j$.
Take $x_{j,k} \in E_{j,k}$.
It is then reasonable to expect that, for any $\epsilon > 0$, if
$b$ is sufficiently small, and if $x_{j,k} \in E_{j,k}$, then
\begin{equation}
\label{kerfajap}
(A_a-\epsilon) \langle F,F \rangle  \leq
\sum_j \sum_{k}\left|\int K_{a^j}(x_{j,k},y) F(y)d\mu(y)\right|^2 \mu(E_{j,k}) \leq (B_a + \epsilon)  \langle F,F \rangle ,
\end{equation}
which means
\begin{equation}
\label{mtvest}
(A_a-\epsilon) \langle F,F \rangle  \leq
\sum_j \sum_{k} |(F,\phi_{j,k})|^2 \leq (B_a + \epsilon)  \langle F,F \rangle ,
\end{equation}
where
\begin{equation}
\label{phjkdf}
\phi_{j,k}(x)= \mu(E_{j,k})^{1/2}~ \overline{K_{a^j}}(x_{j,k},x),
\end{equation}
or, formally, from (\ref{schkerf}), (since $\overline{f}(a^{2j}T) = f(a^{2j}T)^\ast$),
\[ \phi_{j,k}(x)=  \mu(E_{j,k})^{1/2} \left[\overline{f}(a^{2j}T)\delta_{x_{j,k}}\right](x).\]

We will show that (\ref{mtvest}) indeed holds,
provided the $E_{j,k}$ are also ``not too small'' (precisely, if they satisfy (\ref{measgeq0}) directly below).
In fact, in Theorem \ref{framain}, we shall show (a more general form of) the following result:

\begin{theorem}
\label{framain0}
Fix $a >1$, and say $c_0 , \delta_0 > 0$.  Suppose $f \in {\mathcal S}(\RR^+)$,
and $f(0) = 0$.  Suppose that the Daubechies condition $(\ref{daub})$ holds.
Then there exists a constant $C_0 > 0$ $($depending only on ${\bf M}, f, a, c_0$ and $\delta_0$$)$
as follows:\\
For $t > 0$, let $K_t$ be the kernel of $f(t^2\Delta)$.
Say $0 < b < 1$.
Suppose that, for each $j \in \ZZ$, we can write ${\bf M}$ as a finite disjoint union of measurable sets
$\{E_{j,k}: 1 \leq k \leq N_j\}$,
where:
\begin{equation}
 \label{diamleq0}
\mbox{the diameter of each } E_{j,k} \mbox{ is less than or equal to } ba^j,
\end{equation}
and where:
\begin{equation}
 \label{measgeq0}
\mbox{for each } j \mbox{ with } ba^j < \delta_0,\: \mu(E_{j,k}) \geq c_0(ba^j)^n.
\end{equation}
$($Such $E_{j,k}$ exist
provided $c_0$ and $\delta_0$ are sufficiently small, independent of the
values of $a$ and $b$.$)$

For $1 \leq k \leq N_j$, define $\phi_{j,k}$ by (\ref{phjkdf}).  Then
if $P$ denotes the projection in $L^2({\bf M})$ onto the space of constants, we have
\[(A_a-C_0b) \langle F,F \rangle  \leq
\sum_j \sum_{k} |(F,\phi_{j,k})|^2 \leq (B_a + C_0b)  \langle F,F \rangle ,\]
for all $F \in (I-P)L^2({\bf M})$.  In particular, if $A_a - C_0b > 0$, then
$\left\{\phi_{j,k}\right\}$ is a frame for $(I-P)L^2({\bf M})$,
with frame bounds $A_a - C_0b$ and $B_a + C_0b$.
\end{theorem}

Thus, in these circumstances, if $b$ is sufficiently small,
$\{\phi_{j,k}\}$ is a frame, in fact a nearly tight frame, since
\[ \frac{B_a + \epsilon}{A_a - \epsilon} \sim \frac{B_a}{A_a} =
1 + O\left(|(a-1)^2 (\log|a-1|)|\right). \]

Most of the work in this article will be in justifying the passage from
(\ref{kerfaj}) to (\ref{kerfajap}).  We will shortly sketch how this is done.
First, however, we should explain in what sense the $\phi_{j,k}$
look like a {\em wavelet} frame.  We recall that, on the real line, wavelets
are obtained through dyadic translations and dilations from one fixed function.

Let us first, then, explain in what sense the passage from one $j$ to another
is analogous to dilation.  For this, we need to understand very precisely the
behavior of $K_t(x,y)$ near the diagonal.

If we were on $\RR^n$, $K_t(x,y)$, the kernel of $f(t^2\Delta)$, would be of the
form $t^{-n} \psi((x-y)/t)$ for some $\psi \in \mathcal{S}$.
(Here $\hat{\psi} = G$, where $G(\xi) = f(\xi^2)$.)
For any $N, \alpha, \beta$,
there would thus exist $C_{N, \alpha, \beta}$ such that
\begin{equation}\notag
t^{n+|\alpha|+|\beta|} \left|\left(\frac{x-y}{t}\right)^N \partial_x^{\alpha} \partial_y^{\beta} K_t(x,y)\right|
\leq C_{N, \alpha, \beta}
\end{equation}
for all $t, x, y$.

In Lemma \ref{manmolold} of \cite{gmcw}, we showed that one has similar estimates on ${\bf M}$,
so that, in effect, $K_t(x,y)$ ``behaves like''
$t^{-n} \psi((x-y)/t)$ near the diagonal $D$ (for some $\psi$).  Precisely, we
showed:
\begin{lemma}
\label{manmol}
For every pair of
$C^{\infty}$ differential operators $X$ $($in $x)$ and $Y$  $($in $y)$ on ${\bf M}$,
and for every integer $N \geq 0$, there exists $C_{N,X,Y}$ as follows.  Suppose
$\deg X = j$ and $\deg Y = k$.  Then
\begin{equation}
\label{diagest}
t^{n+j+k} \left|\left(\frac{d(x,y)}{t}\right)^N XYK_t(x,y)\right| \leq C_{N,X,Y}
\end{equation}
for all $t > 0$ and all $x,y \in {\bf M}$.
\end{lemma}

(Here $d$ is geodesic distance on ${\bf M}$.)
In this sense, then, by (\ref{phjkdf}), passage from one $j$ to another (for the $\phi_{j,k}$) is analogous
to dilation.

It remains to discuss if there is any sense in which the passage from one $k$ to
another (for fixed $j$) is analogous to translation.  A general manifold ${\bf M}$ has nothing
akin to translations, but in section 6 of \cite{gmcw}, we discussed
the situation in which ${\bf M}$ has a
transitive group $G$ of smooth metric isometries, and in that case there is a satisfying result.
(Manifolds with such a group $G$ are usually called {\em homogeneous}.  Obvious examples are the sphere and the torus.)
In \cite{gmcw}, we observed that $K_t(Tx,Ty) = K_t(x,y)$ for all $T \in G$ and $x,y \in {\bf M}$.
Using (\ref{phjkdf}), let us write
$\phi_{j,k}= \mu(E_{j,k})]^{1/2} \varphi_{j,k}$
where $\varphi_{j,k}(x) =  \overline{K_{a^j}}(x_{j,k},x)$.
Then,
for any fixed $j$, we see at once that one can obtain all of the $\varphi_{j,k}$
by applying elements
of the group $G$ to any one of them. For example, on the sphere, for any fixed
$j$, all of the $\varphi_{j,k}$ are rotates of each other.  This is
a natural analogue for ${\bf M}$ of the usual requirement on $\RR^n$ that all
elements of the frame at a particular scale $j$ be translates of each other.  \\
\ \\
Besides its pleasant and suggestive appearance, Lemma \ref{manmol} (as stated above)
is actually crucial, in enabling us to rigorously pass from
(\ref{kerfaj}) to (\ref{kerfajap}).  Let us explain why.  We fix a finite set
${\mathcal P}$ of real $C^{\infty}$ vector fields on ${\bf M}$, whose
elements span the tangent space at each point.  We let
${\mathcal P}_0 = {\mathcal P} \cup \{\mbox{ the identity map}\}$.

Let $C_0^{\infty}({\bf M})$ equal the set of smooth $\varphi$ on ${\bf M}$ with
integral zero.   For $x \in {\bf M}$ and $t > 0$, we define
$$ {\mathcal M}_{x,t} = \left\{\varphi \in C_0^{\infty}({\bf M}):
t^{n+\deg Y} \left|\left(\frac{d(x,y)}{t}\right)^N Y\varphi(y)\right|\leq 1
\mbox{ whenever } y \in {\bf M}, \: 0 \leq N \leq n+2 \mbox{ and }
Y \in {\mathcal P}_0\right\}.$$

For example, if we define $\varphi^t_x(y) = K_t(x,y)$, then
by Lemma \ref{manmol},
there exists $C > 0$ such that $\varphi^t_x$ is in $C{\mathcal M}_{x,t}$
for all $x \in {\bf M}$ and all $t > 0$.
(The space ${\mathcal M}_{x,t}$ is a variant of
a space of molecules, as defined earlier in \cite{gil} and \cite{han2}.)

Our main technical tool will be the following fact.   For each $j \in \ZZ$,
we write ${\bf M}$ as a finite disjoint union of sets $E_{j,k}$ of diameter at most $a^j$,
and we choose points $x_{j,k} \in E_{j,k}$ for all $j,k$.  We then choose functions
$\varphi_{j,k}, \psi_{j,k} \in
{\mathcal M}_{x_{j,k},a^j}$ for all $j,k$.  We then claim that
the ``summation operator'' $S$ defined by
\begin{equation}
\label{sumopint}
SF = \sum_{j}\sum_k \mu(E_{j,k}) \langle  F, \varphi_{j,k}  \rangle  \psi_{j,k}
\end{equation}
is bounded on $L^2({\bf M})$, and that $\|S\| \leq C$ where $C$ is {\em independent}
of our choices of $E_{j,k}, ~\varphi_{j,k}, ~\psi_{j,k}$.  A precise statement is
given in Theorem \ref{sumopthm}.  To prove this result, we use the $T(1)$ theorem to
show that $S$ is a Calder\'on-Zygmund operator.  (If ${\bf M}$ were $\RR^n$,
and if the $\varphi_{j,k}$ and $\psi_{j,k}$ were obtained from fixed functions by
means of translations and dilations in the usual manner, then this result is well known
(see e.g. \cite{gil}).)

This result shows that, at least, the sum in (\ref{kerfajap}) (or equivalently
(\ref{mtvest})) is finite, since it equals $ \langle  SF,F  \rangle $, provided, in
the formula (\ref{sumopint}), we set
\begin{equation}
\label{vphpsk}
\varphi_{j,k}(y) = \psi_{j,k}(y) = \overline{K}_{a^j}(x_{j,k},y).
\end{equation}

Once this is known, here is the spirit of how one can pass from
(\ref{kerfaj}) to (\ref{kerfajap}).  One can note that, by the
fundamental theorem of calculus, the error made when replacing the outer integral in
(\ref{kerfaj}) by the Riemann sum in (\ref{kerfajap}) can be expressed in
terms of the integrals of appropriate derivatives over short intervals.  Theorem
\ref{sumopthm} can then be used again to bound the integrand in this expression
for the error, and the error itself
is then small (in fact is $O(b)$) because the intervals of integration are short.
(It will actually be convenient to use identities such as $b\int_0^1 H(bs) ds =
\int_0^b H(s) ds$ to convert the integrals over short intervals to integrals
over $[0,1]$; then the desired $b$ will actually multiply the integrals.)  The
precise argument will be given in Theorem \ref{framain}.\\

A major advantage of working with functions of $\Delta$ is that one can now
perform a space-frequency analysis.  Based upon how well-localized a function
$F \in L^2$ is in space and in frequency, we can describe which terms in
the summation
\begin{equation}\notag
SF = \sum_j \sum_k \mu(E_{j,k})  \langle  F,\varphi_{j,k}  \rangle  \varphi_{j,k}
\end{equation}
are so small that they can be neglected.  (Here, $\varphi_{j,k}$ is again as in (\ref{vphpsk}).)
In the analogous situation on
the real line (\cite{Daubechies92}, Chapter 2), one may do this by means of a
time-frequency analysis.  The analogue here is to do a ``space-frequency'' analysis
(it is hardly reasonable to think of ${\bf M}$ as parametrizing time!).
However, in contrast
to the analysis for $\RR$ in \cite{Daubechies92}, we use the spectral theorem.  We shall
first do a frequency analysis. (Assuming, roughly, that $F$ is well-localized in
frequency, we describe which $j$ can be neglected).  This will reduce us to a finite sum.
Then we will do a spatial analysis. (Assuming, roughly, that $F$ is well-localized in
space, we describe which $k$ can be neglected for each of our now finite set of $j$).

Specifically, let $0 = \lambda_0 < \lambda_1 \leq \ldots$ be the eigenvalues of
$\Delta$ on ${\bf M}$, and let $\{u_m: 0 \leq m < \infty\}$ be an orthonormal
basis of eigenvectors with $\Delta u_m = \lambda_m u_m$ for each $m$.  If
$F = \sum c_m u_m \in L^2({\bf M})$, and $L \geq 0$, let $P_{[0,L]} F
= \sum_{\lambda_m \leq L} c_m u_m$.  Also let $J \geq 1$ be an integer, and suppose that
$F$ vanishes to order $l \geq 1$ at $0$.
In Theorem \ref{freqthm} we shall obtain the estimate
\begin{equation}
\label{freqestint}
\|SF - \sum_{j=-M}^N \sum_k \mu(E_{j,k})  \langle  F,\varphi_{j,k}  \rangle  \varphi_{j,k}\|_2
\leq \left(C_0b+ \frac{c^{\prime}_L}{a^{4Ml}}  + \frac{C_J^{\prime}}{a^{4NJ}}\right)\|F\|_2 +
2B_a \|(I - P_{[0,L]})F\|_2,
\end{equation}
where the constant $c^{\prime}_L$ depends only on $L$, $f$, and $a$,
the constant $C_J^{\prime}$ depends only on $J$, $f$, $a$ and ${\bf M}$, and where
$B_a$ is an upper frame bound for the frame $\{ \mu(E_{j,k})^{1/2} \varphi_{j,k}\}$.
($c^{\prime}_L$ and $C_J^{\prime}$ are identified in Theorem \ref{freqthm}.)

Thus, say one wants to compute $SF$ to a certain precision.  One calculates
the finite sum
\begin{equation}
\label{MNappS1}
\sum_{j=-M}^N \sum_k \mu(E_{j,k})  \langle  F,\varphi_{j,k}  \rangle  \varphi_{j,k}
\end{equation}
for $M, N$ sufficiently large;
how large must one take them to be?  One first chooses $b, J, L, N$ so that the
first, third and fourth terms on the right side of (\ref{freqestint}) are
very small.  Then one chooses $M$ to make the second term on the right side of
(\ref{freqestint}) very small as well.

Let us explain, briefly, why one ought to expect a result like (\ref{freqestint}).
Note that (\ref{mtvest}) says that, if $F = (I-P)F$, then
\begin{equation}\notag
(A_a-\epsilon) \langle F,F \rangle  \leq  \langle SF,F \rangle  \leq (B_a + \epsilon)  \langle F,F \rangle .
\end{equation}
Thus, the arguments leading from (\ref{aasumba0}) to (\ref{mtvest}) show that $$SF \sim
\sum_{j=-\infty}^{\infty} |f|^2(a^{2j}\Delta) F.$$
  In turn, by (\ref{gelmay2}) and
(\ref{daubest}), $\sum_{j=-\infty}^{\infty} |f|^2(a^{2j}\Delta) F \sim C_0F$, where $C_0 =
c/(2 |\log a|)$.  Similar reasoning shows that the truncated sum

\[ \sum_{j=-M}^N \sum_k \mu(E_{j,k})  \langle  F,\varphi_{j,k}  \rangle  \varphi_{j,k} \sim
\sum_{j=-M}^N |f|^2(a^{2j}\Delta) F. \]

Let us then try to understand why, if $F$ is well-localized in frequency, then
$\sum_{j=-M}^N |f|^2(a^{2j}\Delta) F \sim C_0F$ for $M,N$ sufficiently large.  That would then
strongly suggest that we can truncate the series for $SF$ with only a small error, as in
(\ref{freqestint}).
 For $\lambda > 0$, let
\begin{equation}
\label{gladf}
g(\lambda) = \sum_{j=-\infty}^{\infty} |f(a^{2j} \lambda)|^2,
\end{equation}
  the Daubechies sum, which is always between $A_a$ and $B_a$, for any $\lambda$, although
(since $f$ vanishes at $0$ and $\infty$) the
set of terms in the series which are sizable {\em depends very much on} $\lambda$.
Let
\begin{equation}
\label{gmnladf}
g_{M,N}(\lambda) = \sum_{j=-M}^N |f(a^{2j} \lambda)|^2.
\end{equation}
 If  $F = \sum_{m \geq 1} c_m u_m \in L^2({\bf M})$, then
\begin{equation}
\label{frqstup}
\sum_{j=-M}^N |f|^2(a^{2j} T)F = \sum_{m=1}^{\infty} c_m\left [ \sum_{j=-M}^N |f|^2(a^{2j} T)u_m\right] =
  \sum_{m=1}^{\infty} c_m g_{M,N}(\lambda_m)u_m.
\end{equation}
If we had the Daubechies sum $g$ here in place of $g_{M,N}$, this would be a good
approximation to $C_0F$.  In the actual situation, say $F$ is well-localized, in the sense that
all $c_m$ with $\lambda_m > L$ are {\em negligible}.  Then the rightmost member of
(\ref{frqstup}) will be a good approximation to
$C_0F$ provided
\[ g_{M,N}(\lambda_m) \sim g(\lambda_m) \]
whenever $\lambda_1 \leq \lambda_m \leq L$.  Recalling (\ref{gladf}) and (\ref{gmnladf}), and the fact that $f$
vanishes at $0$ and vanishes rapidly at $\infty$,
we see that this happens provided that the interval $[-M,N]$ contains all $j$ with $a^{2j} \lambda_m
\sim 1$ whenever $\lambda_1 \leq \lambda_m \leq L$, or in other words, all $j$ with $C/L \leq a^{2j} \leq
C/\lambda_1$.  This will happen if $M$ and $N$ are sufficiently large.  Thus, we expect that we can truncate
with only a small error, as is precisely borne out in
(\ref{freqestint}).\\

(\ref{freqestint}) reduces the approximate evalustion of $SF$ to the calculation of the
finite sum (\ref{MNappS1}).  It may well be possible to further reduce
the number of terms that need to be considered, provided
$F$ is well-localized in space, e.g. if it is supported in a ball $B$ of small radius.
For, say $x_{j,k}$ is far from this ball.   By Lemma \ref{manmol},
for any $I > 0$, there exists
$C_I$ such that $|\varphi_{j,k}(y)| \leq C_Ia^{(I-n)j}d(x_{j,k},y)^{-I}$
for all $y \in B$, which is small; so $| \langle  F, \varphi_{j,k}  \rangle |$
is small as well.  In Theorem \ref{spacest}, we make these considerations
precise, thereby performing a ``spatial analysis'' as well.\\

Our frames should be viewed as a discretization of the continuous wavelets we presented in our earlier
article \cite{gmcw}.  However, it is important to notice that we will not actually make use of those
continuous wavelets, since we get better results by discretizing Calder\'on's integral first (to obtain Daubechies'
condition), and then discretizing as in (\ref{aasumba0}) -- (\ref{mtvest}).  In this manner, we obtain
frames if Daubechies' condition holds, and nearly tight frames as $a \rightarrow 1^+$, with nearly quadratic
convergence of the ratio of the frame bounds to $1$ -- results that we would not obtain if we were
to discretize our continuous wavelets in a more standard manner.  (As an aside, we note that if
one did want to discretize the consinuous ${\cal S}$-wavelets of \cite{gmcw} in a standard manner, one would
want to require an additional condition such as (\ref{xykt}) of that article, in order
to be able to apply the fundamental theorem of calculus in the $t$ variable.  As we noted in \cite{gmcw},
this additional condition does hold if $K_t$ is the Schwartz kernel of $f(t^2T)$.)\\

In another article (already available \cite{gm3}), we show that one can determine whether
$F$ is in a Besov space, solely from a knowledge of the size of its frame coefficients.
In a future article, we hope to study the same question for Triebel-Lizorkin spaces.
(The analogous problems on $\RR^n$ were solved in \cite{fj} and \cite{fj2}.)

\section{Frames}

We shall need the following basic facts, from Section 3 of \cite{gmcw}, about ${\bf M}$ and its
geodesic distance $d$.
For $x \in {\bf M}$, we let $B(x,r)$ denote the ball $\{y: d(x,y) < r\}$.

\begin{proposition}
\label{ujvj}
Cover ${\bf M}$ with a finite collection of open sets $U_i$  $(1 \leq i \leq I)$,
such that the following properties hold for each $i$:
\begin{itemize}
\item[$(i)$] there exists a chart $(V_i,\phi_i)$ with $\overline{U}_i
\subseteq V_i$; and
\item[$(ii)$] $\phi_i(U_i)$ is a ball in $\RR^n$.
\end{itemize}
Choose $\delta > 0$ so that $3\delta$ is a Lebesgue number for the covering
$\{U_i\}$.  Then, there exist $c_1, c_2 > 0$ as follows:\\
For any $x \in {\bf M}$, choose any $U_i \supseteq B(x,3\delta)$.  Then, in
the coordinate system on $U_i$ obtained from $\phi_i$,
\begin{equation}
\label{rhoeuccmp2}
d(y,z) \leq c_2|y-z|
\end{equation}
for all $y,z \in U_i$; and
\begin{equation}
\label{rhoeuccmp}
c_1|y-z| \leq d(y,z)
\end{equation}
for all $y,z \in B(x,\delta)$.
\end{proposition}

We fix collections $\{U_i\}$, $\{V_i\}$, $\{\phi_i\}$ and also $\delta$ as in
Proposition \ref{ujvj}, once and for all.
\begin{itemize}
\item Notation as in Proposition \ref{ujvj},
there exist $c_3, c_4 > 0$, such that, whenever $x \in {\bf M}$ and $0 < r \leq \delta$,
\begin{equation}
\label{ballsn}
c_3r^n \leq \mu(B(x,r)) \leq c_4r^n
\end{equation}
\item
For any $N > n$ there exists $C_N$ such that, for all $x \in {\bf M}$ and $t > 0$,
\begin{equation}
\label{ptestm}
\int_{\bf M} [1 + d(x,y)/t]^{-N} d\mu(y) \leq C_N t^n
\end{equation}
\item
For any $N > n$ there exists $C_N^{\prime}$ such that, for all $x \in {\bf M}$ and $t > 0$
\begin{equation}
\label{ptestm1}
\int_{d(x,y) \geq t} d(x,y)^{-N} d\mu(y) \leq C_N^{\prime} t^{n-N}
\end{equation}
\item
For any $N > n$ there exists $C_N^{\prime \prime}$ such that for all $x,y \in {\bf M}$ and $t > 0$,
\begin{equation}
\label{ptestm2}
\int_{\bf M} [1 + d(x,z)/t]^{-N} [1 + d(z,y)/t]^{-N}d\mu(z) \leq
C_N^{\prime \prime} t^n[1 + d(x,y)/t]^{-N}
\end{equation}
\item
For all $M, t > 0$, and for all
$E \subseteq {\bf M}$ with diameter less than $Mt$, if
$x_0 \in E$, then one has that
\begin{equation}
\label{alcmpN}
\frac{1}{M+1}[1+d(x,y)/t] \leq [1+d(x_0,y)/t] \leq (M+1)[1+d(x,y)/t]
\end{equation}
for all $x \in E$ and all $y \in {\bf M}$.\\
\end{itemize}
\ \\
Now, fix a finite set ${\mathcal P}$ of real $C^{\infty}$ vector fields on ${\bf M}$, whose
elements span the tangent space at each point.  Using compactness,
it is then easy to see that, if
${\mathcal P}'$ is another such set, there is a constant $C > 0$ such that
for any $f \in C^1({\bf M})$, and any $x \in {\bf M}$, we have that
\begin{equation}\notag
\max_{X' \in {\mathcal P}'} |X'f(x)| \leq C\max_{X \in {\mathcal P}} |Xf(x)|.
\end{equation}
As will soon be apparent, this implies that the choice of set ${\mathcal P}$ is immaterial
(up to a constant) for the estimates which follow.  For convenience, we
choose ${\mathcal P}$ so that for every $l$ with $1 \leq l \leq n$, and every $i$, there
is a vector field $X$ in ${\mathcal P}$ such that $X \equiv \partial/\partial x_l$
on $U_i$ (in the local coordinates obtained from $\phi_i$).  (Notation as
in Proposition  \ref{ujvj}).
We also let
\[ {\mathcal P}_0 = {\mathcal P} \cup \{\mbox{ the identity map}\}.\]
%
%
 Let
\[ C_0^{\infty}({\bf M}) = \left\{\varphi \in C^{\infty}({\bf M}):\;  \int_{\bf M} \varphi d\mu = 0\right\}. \]
For any $x \in {\bf M}$, we let
\begin{equation}
\label{mxdef}
{\mathcal M}_{x,t} = \left\{\varphi \in C_0^{\infty}({\bf M}):  \;
t^{n+\deg Y} \left|\left(\frac{d(x,y)}{t}\right)^N Y\varphi(y)\right|\leq 1
\mbox{ whenever } y \in {\bf M}, \: 0 \leq N \leq n+2 \mbox{ and }
Y \in {\mathcal P}_0\right\}.
\end{equation}
{\bf Important Example} Notation as in Lemma \ref{manmol}, for each $x \in {\bf M}$, we
define the functions $\varphi^t_x, \psi^t_x$ on ${\bf M}$ by
$\varphi^t_x(y) = K_t(x,y)$ and $\psi^t_x(y) = K_t(y,x)$.  Then, by Lemma \ref{manmol}
there exists $C > 0$ such that $\varphi^t_x$ and $\psi^t_x$ are in $C{\mathcal M}_{x,t}$
{\em for all} $x \in {\bf M}$ and all $t > 0$.\\
\ \\
Our main task in this section is to justify passing from (\ref{kerfaj}) to (\ref{kerfajap}).
Our main technical tool will be the following fact.   For each $j \in \ZZ$,
we write ${\bf M}$ as a finite disjoint union of sets $E_{j,k}$ of diameter at most $a^j$,
and we choose points $x_{j,k} \in E_{j,k}$ for all $j,k$.  We then choose functions
$\varphi_{j,k}, \psi_{j,k} \in
{\mathcal M}_{x_{j,k},a^j}$ for all $j,k$.  We then claim that
the ``summation operator'' $S$ defined by
\begin{equation}
\label{sumopdf}
SF = \sum_{j}\sum_k \mu(E_{j,k}) \langle  F, \varphi_{j,k}  \rangle  \psi_{j,k}
\end{equation}
is bounded on $L^2({\bf M})$, and that $\|S\| \leq C$ where $C$ is {\em independent}
of our choices of $E_{j,k}, ~\varphi_{j,k}, ~\psi_{j,k}$.  (A precise statement is
given in Theorem \ref{sumopthm} below.)

If we
recall the definition of the space ${\mathcal M}_{x,t}$, in (\ref{mxdef}), we see:
\begin{itemize}
\item
For every $C_1 > 0$ there exists $C_2 > 0$ such that whenever
$t > 0$ and $d(x,y) \leq C_1t$,
we have
\begin{equation}
\label{nrbyok}
{\mathcal M}_{y,t} \subseteq C_2{\mathcal M}_{x,t}.
\end{equation}
Say now $x \in {\bf M}$, and $R > 0$.  We say that a function $\omega \in
C^1({\bf M})$ is a {\em bump function for the ball} $B(x,R)$, provided $\supp\; \omega
\subseteq B(x,R)$, $\|\omega\|_{\infty} \leq 1$ and $\max_{X \in {\mathcal P}} \|X\omega\|_{\infty}
\leq 1/R$.  If  $\omega$ is a bump function for some ball $B(x,R)$,
we say that $\omega$ is an {\em $R$-bump function}.  (This notion is modelled on the
following situation on $\RR^n$: take a $C^1$ function $\varphi$ on $\RR^n$ with support in the unit
ball and with $\|\varphi\|_{C^1} \leq 1$, and look at $\omega(x) = \varphi(x/R)$.)  Note
also that {\em every} $C^1$ function on ${\bf M}$ is a multiple of an $R$-bump function,
where $R = 2(\diam{\bf M})$.
\item
 There is a constant $C > 0$ such that for every $R \leq 2(\diam{\bf M})$,
every $R$-bump function $\omega$, and every $x,y \in {\bf M}$, one has
\begin{equation}
\label{mvtm}
|\omega(x)-\omega(y)| \leq Cd(x,y)/R.
\end{equation}
{\em To see this, we again use the notation of Proposition \ref{ujvj}.
If $d(x,y) < \delta$, we choose
$U_i \supseteq B(x,3\delta)$, and use the mean value theorem in the local coordinates
on $U_i$ obtained from $\phi_i$, to obtain (\ref{mvtm}).
If, on the other hand, $d(x,y) > \delta$,
we simply note
$$|\omega(x)-\omega(y)| \leq 2 = [2R/d(x,y)][d(x,y)/R]
\leq [4(\diam{\bf M})/\delta][d(x,y)/R].$$}
\end{itemize}

If $T: C^{1}({\bf M}) \rightarrow L^2({\bf M})$ is linear,
we say that a linear operator $T^* : C^1({\bf M}) \rightarrow L^2({\bf M})$
is its formal adjoint if for all $f, g \in C^{1}$ we have
\begin{equation}\notag
<Tf,g> = <f,T^*g>.
\end{equation}
$T^*$ is evidently unique if it exists.

Recall that ${\mathcal P}_0 = {\mathcal P} \cup \{\mbox{ the identity map}\}$.
We will be using the following form of the David-Journ\'e $T(1)$ theorem
\cite{DavidJourne84}.

\begin{theorem} \label{tof1q}
There exist $C_0, N > 0$, such that for any $A > 0$, we have the following.
Whenever $T: C_c^{1}({\bf M}) \rightarrow L^2({\bf M})$ has a formal adjoint
$T^*: C_c^{1}({\bf M}) \rightarrow L^2({\bf M})$, and whenever $T, T^*$ satisfy:
\begin{itemize}
\item[$(i)$] $\|T \omega\|_2 \leq A R^{n/2}$ and $\|T^* \omega\|_2 \leq A R^{n/2}$
for all $R$-bump functions $\omega$;
\item[$(ii)$] There is a kernel $K(x,y)$, $C^1$ off the diagonal,
such that if $F \in C^{1}$,
then for $x$ outside the support of $F$, $(TF)(x) = \int K(x,y)F(y)dy$; and
\item[$(iii)$] Whenever $X$ $($acting in the $x$ variable$)$ and $Y$ (acting in the
$y$ variable) are in ${\mathcal P}_0$, and at least one of $X$ and $Y$ is the identity map, we have
\[ |XYK(x,y)| \leq A\;d(x,y)^{-(n+\deg X + \deg Y)} \]
for all $x, y \in {\bf M}$ with $x \neq y$; and
\item[$(iv)$]$T(1) = T^*(1) = 0$,
\end{itemize}
then $T$ extends to a bounded operator on $L^2(\bf M)$, and $\|T\| \leq C_0 A$.
\end{theorem}

A simple proof of this theorem in the case ${\bf M} = \RR^n$ may be found in \cite{fabes}.
That proof immediately adapts to our situation, in which ${\bf M}$ is a smooth compact
oriented Riemannian manifold.\\
\ \\
{\bf Note:} In verifying condition (i) of Theorem \ref{tof1q}, one may assume that
$R \leq 2(\diam{\bf M})$.  For, set $R_0 = 2(\diam{\bf M})$; it is evident that if
$R \geq R_0$, then any $R$-bump function $\omega$ is also an $R_0$-bump function.  If we know that
$\|T \omega\|_2 \leq A R_0^{n/2}$, then we also know that $\|T \omega\|_2 \leq A R^{n/2}$.\\

We can now prove the main facts about the summation operator (defined in (\ref{sumopdf})).

\begin{theorem}
\label{sumopthm}
Fix $a>1$.  Then there exists $C_1, C_2 > 0$ as follows.

For each $j \in \ZZ$, write ${\bf M}$ as a finite disjoint union of
measurable subsets $\{E_{j,k}: 1 \leq k \leq N_j\}$, each of diameter less than $a^j$.
For each $j,k$, select any $x_{j,k} \in E_{j,k}$, and select $\varphi_{j,k}$,
 $\psi_{j,k}$ with \\$\varphi_{j,k} , \psi_{j,k}\in \mathcal{M}_{x_{j,k}, a^j}$.
For $F \in C^1({\bf M})$, we claim that we may define
\begin{equation}
\label{sumopdf2}
SF = S_{\{\varphi_{j,k}\},\{\psi_{j,k}\}}F =
\sum_{j}\sum_k \mu(E_{j,k}) \langle  F, \varphi_{j,k}  \rangle  \psi_{j,k}.
\end{equation}
$($Here, and in similar equations below, the sum in $k$ runs from $k = 1$ to $k = N_j$.$)$
Indeed:
\begin{itemize}
\item[$(a)$] For any $F\in C^1({\bf M})$, the series defining $SF$ converges absolutely, uniformly on $\bf{M}$,
In fact, if, for some $R\leq 2(\diam{\bf M})$, $F = \omega$ is an $R$-bump function, then the sum of
the absolute values of
the terms is less than or equal to $C_1$ at each point of ${\bf M}$.
\item[$(b)$] $\parallel SF\parallel_2\leq C_2\parallel F\parallel_2$ for all
 $F\in C^1(\bf{M})$.\\
Consequently, $S$ extends to be a bounded operator on $L^2({\bf M})$, with norm less than
or equal to $C_2$.  In fact, if we put $T = S$, then $S$ satisfies the hypotheses of
Theorem $\ref{tof1q}$.
\item[$(c)$] If $F \in L^2({\bf M})$, then
\begin{equation}\notag
SF = \sum_{j}\sum_k \mu(E_{j,k}) \langle  F, \varphi_{j,k}  \rangle  \psi_{j,k}
\end{equation}
where the series converges unconditionally.
\item[$(d)$] If $F,G \in L^2({\bf M})$, then
\begin{equation}\notag
 \langle  SF,G  \rangle =
\sum_{j}\sum_{k}\mu(E_{j,k}) \langle  F, \varphi_{j,k} \rangle   \langle \psi_{j,k},G \rangle ,
\end{equation}
where the series converges absolutely.\\
\end{itemize}
\end{theorem}

\begin{proof}[Proof of (a)]
We first prove the second statement of (a).
For each $j$, and each
$R \leq 2(\diam{\bf M})$,  we let
\[ C_{j,R} = \sup \left|  \langle  \omega, \varphi_{j,k} \rangle \right| \]
where the sup is taken over all $R$-bump functions $\omega$ and all $k$.
Since $ \left| \langle  \omega, \varphi_{j,k} \rangle \right| \leq \|\omega\|_1\|\varphi_{j,k}\|_{\infty}$
we surely have that

\begin{equation}
\label{cjrest1}
C_{j,R} \leq CR^n a^{-jn}.
\end{equation}

On the other hand, we claim that if $a^j \leq R$, then
\begin{equation}
\label{cjrest2}
C_{j,R} \leq Ca^j/R.
\end{equation}
In fact, since  $\int \varphi_{j,k}d\mu=0$, we find

 \begin{align}\label{goholder}\left| \langle  \omega,  \varphi_{j,k} \rangle \right|&= \left| \int
 \omega(x)\overline{\varphi_{j,k}(x)}d\mu(x)\right|  \\\notag
&=
  \left| \int \left( \omega(x)-\omega(x_{j,k})\right)\overline{\varphi_{j,k}(x)}d\mu(x)\right|\\\notag
  &\leq \int \left|  \omega(x)-\omega(x_{j,k})\mid \mid\varphi_{j,k}(x)\right| d\mu(x)\\\notag
 &\leq C\int R^{-1}d(x,x_{j,k}) \left|\varphi_{j,k}(x)\right| d\mu(x)\\\notag
&\leq CR^{-1} a^{-jn}a^j\int \left[1+d(x,x_{j,k})/a^j\right]^{-n-1}d\mu(x)\\\label{endholder}
&\leq Ca^j/R
  \end{align}
by (\ref{ptestm}).  Now, for any $j$ and any $y \in {\bf M}$, we have
\begin{align}
\sum_k \mu(E_{j,k}) \left|  \langle  \omega,  \varphi_{j,k} \rangle \right|\left|  \psi_{j,k}(y)\right|&\leq C_{j,R}\sum_k   \mu(E_{j,k}) \left|  \psi_{j,k}(y)\right|
\label{sumsetup}\\
&\leq C C_{j,R} a^{-jn}\sum_k \mu(E_{j,k}) \left[1+d(y,x_{j,k})/a^j\right]^{-n-1}
\label{sumtogo}
\\
&\leq C  C_{j,R} a^{-jn}\int_{\bf{M}}\left[1+d(y,x)/a^j\right]^{-n-1} d\mu(x) \label{sumgone}\\
&\leq   C  C_{j,R},  \label{ccjrhere}
\end{align}
again by (\ref{ptestm}).  (In passing from (\ref{sumtogo}) to (\ref{sumgone}), we used
(\ref{alcmpN}).)

Taking  the sum over $j$ we have:
\begin{align}\notag
\sum_j\sum_k   \mu(E_{j,k}) \left|  \langle  \omega,  \varphi_{j,k} \rangle \right| \left|
\psi_{j,k}(y)\right|&\leq C   \sum_j   C_{j,R}  \\\notag
& \leq C\left[\sum_{a^j \leq R}  R^{-1}a^j + \sum_{a^j > R}   R^{n}a^{-jn}\right] \\
& \leq C_1, \label{c1here}
\end{align}
by (\ref{cjrest1}) and (\ref{cjrest2}).  This proves the second statement in (a).
Since every $F \in C^1$ is a multiple of an $R$-bump function, with
$R = 2(\diam{\bf M})$, the
first statement in (a) is a consequence of the two calculations above, which ended
with (\ref{ccjrhere}) and (\ref{c1here}).
\end{proof}

 \begin{proof}[Proof of (b)]

To begin, let us set
\[ {\mathcal N}_{x,t} = \left\{\Phi \in C^{\infty}({\bf M}):  \;
t^{n} \left|\left(\frac{d(x,y)}{t}\right)^N \Phi(y)\right|\leq 1
\mbox{ whenever } y \in {\bf M}, \: 0 \leq N \leq n+2\right\} \]
{\bf Claim:} There exists $C_3 > 0$ (independent of our choice of the $E_{j,k}$
and $x_{j,k}$) as follows:
Suppose $\{  \Phi_{j,k}\}$ and $\{  \Psi_{j,k}\}$ are two systems of functions such that
$ \Phi_{j,k},  \Psi_{j,k} \in \mathcal{N}_{x_{j,k}, a^j}$ for all $j,k$, and
suppose $0 \leq J \leq 1$. Then
\begin{equation}\notag
K_J(x,y):=\sum_{j}\sum_k  a^{-jJ}\mu(E_{j,k})  \left| \Phi_{j,k}(x) \Psi_{j,k}(y)\right|
\leq  C_3d(x,y)^{-n-J}
\end{equation}
for any  $x, y\in \bf{M}$, $x\neq y$.\\
\ \\
To prove the claim, note that for any $j$,
\begin{align}\notag
\sum_k  a^{-jJ}\mu(E_{j,k})  \left| \Phi_{j,k}(x) \Psi_{j,k}(y)\right|&\leq
a^{-j(J+2n)}   \sum_k  \mu(E_{j,k})\left[1+d(x,x_{j,k})/a^j\right]^{-(n+2)}\left[1+d(y,x_{j,k})/a^j\right]^{-(n+2)}
\\\notag&\leq  C a^{-j(J+2n)}\int_{\bf M} \left[1+d(x,z)/a^j\right]^{-(n+2)}\left[1+d(y,z)/a^j\right]^{-(n+2)}d\mu(z)\\\notag&\leq  C a^{-j(J+n)} \left[1+d(x,y)/a^j\right]^{-(n+2)}
\end{align}
by (\ref{ptestm2}).
\ \\
Say $x \neq y$ , and let $j_0$ be the integer satisfying $a^{j_0}\leq d(x,y)\leq {a^{j_0+1}}$.
Then  we have

 \begin{align}\notag
 d(x,y)^{J+n}K_J(x,y)  &\leq C \sum_{j} a^{-j(J+n)}     d(x,y)^{J+n}[1+d(x,y)/a^j]^{-n-2}\\\notag
 &\leq  C\left[ \sum_{j\geq j_0}       a^{-j(J+n)} a^{j_0(J+n)} + \sum_{j<j_0}  a^{-j(J+n)} a^{j_0(J+n)} a^{(n+2)(j-j_0)}\right]\\\notag
 &\leq
  C\left[ \sum_{j\geq j_0}     a^{-(j_0-j)(J+n)} + \sum_{j< j_0}   a^{(j_0-j)(J-2)}  \right]\\\notag
&\leq C.
 \end{align}
 since we are assuming that $0 \leq J \leq 1$.  This establishes the claim.\\
\ \\
We return now to (\ref{sumopdf2}).
By definition of ${\mathcal M}_{x,t}$, a function $\varphi \in {\mathcal M}_{x,t}$
if and only if $t^{\deg Y} Y\varphi \in {\mathcal N}_{x,t}$ for every $Y \in {\mathcal P}_0$.
We apply this with $t = a^j$ ($j \in \ZZ$).
It therefore follows from the ``claim'' that, whenever
 $X$ (acting in the $x$ variable) and $Y$ (acting in the
$y$ variable) are in ${\mathcal P}_0$, then
\begin{equation}
\label{drvadup}
\sum_{j}\sum_k  \mu(E_{j,k})  \left| X\psi_{j,k}(x)Y\overline{\varphi}_{j,k}(y) \right|
\leq  C_3d(x,y)^{-(n+\deg X + \deg Y)}.
\end{equation}

This has a number of immediate consequences.  Taking $X$ and $Y$ to be identity maps, we
now see that
\begin{equation}
\label{kersfnd}
K(x,y) = \sum_{j}\sum_k  \mu(E_{j,k}) \psi_{j,k}(x)\overline{\varphi}_{j,k}(y)
\end{equation}
is the kernel of $S$, in the sense that if $F \in C^1$, then
\[ [SF](x) = \int_{\bf M} K(x,y)F(y) d\mu(y) \]
whenever $x$ is outside the support of $F$.  (This is because, for fixed $x$,
the series defining $K$ converges absolutely for $y$ in the support of $F$, and the
sum of the absolute value of the terms is bounded independently of $y$.)
Secondly, since we could, in local
coordinates on a $U_i$, take $X$ in (\ref{drvadup}) to be any $\partial/\partial x_l$
(and $Y$ to be any $\partial/\partial y_l$), we see that $K(x,y)$ is $C^1$ off the
diagonal.  Finally we see that we may bring derivatives
past the summation sign in (\ref{kersfnd}), and thus, for any $X, Y \in {\mathcal P}_0$,
\[ \left|XYK(x,y)\right| \leq C_3 d(x,y)^{-(n+\deg X + \deg Y)} \]
for $x \neq y$. \\

This shows that $S$ satisfies conditions $(ii)$ and $(iii)$ of Theorem \ref{tof1q};
we need to show it satisfies the other conditions.  By (a), it is evident that if
$F,G \in C^1({\bf M})$, then
\begin{equation}\notag
 \langle  SF, G  \rangle  =
\sum_{j}\sum_k \mu(E_{j,k}) \langle  F, \varphi_{j,k}  \rangle
 \langle  \psi_{j,k}, G  \rangle ,
\end{equation}
where the sum converges absolutely.
Consequently, the formal adjoint of $S$ is $S^*$, where, if $G \in C^1$,
\[ S^*G = \sum_{j}\sum_k \mu(E_{j,k}) \langle  G, \psi_{j,k}  \rangle  \varphi_{j,k}. \]
Surely $S1= S^*1 = 0$, since $\int \varphi_{j,k} = \int \psi_{j,k} = 0$ for all $j,k$.

That leaves only condition $(i)$ of Theorem \ref{tof1q}.  Suppose then that
$\omega$ is a bump function for the ball $B(x_0,R)$,
where $R \leq 2(\diam{\bf M})$.  If $d(x,x_0) \geq 2R$, then

\begin{align}\notag
\left|[S \omega ](x)\right|&\leq \int \left| K (x,y)\omega(y)\right|d\mu(y)\\\notag
&\leq C \int_{d(y,x_0) \leq R} d(x,y)^{-n}dy  \\\notag
&\leq CR^n\max_{d(y,x_0)  \leq R} d(x,y)^{-n}  \\\notag
&\leq CR^n d(x,x_0)^{-n}.  \end{align}
Recalling (a), we now see that

\begin{align}\notag
\|S\omega\|_2^2 &= \int_{d(x,x_0)<2R} |[S \omega ](x)|^2 d\mu(x)
+ \int_{d(x,x_0) \geq 2R}  |[S \omega ](x)|^2 d\mu(x)\\\notag
&\leq
C\left[(2R)^n+  R^{2n}\int_{d(x,x_0) \geq 2R} d(x,x_0)^{-2n}d\mu(x)\right]\\\notag
&\leq
C[ R^n + R^{2n} R^{-n} ]= CR^n,
\end{align}
by (\ref{ptestm1}).  Hence
\begin{align}\notag
\|S\omega\|_2\leq   CR^{n/2},
\end{align}
This establishes condition $(i)$ of Theorem \ref{tof1q}, and completes
the proof of part (b).
\end{proof}
\begin{proof}[Proof of (c)]
It is evident, by (a), that (c) holds for $F \in C^1({\bf M})$.

Suppose now that $\mathcal{F}\subset \ZZ^2$ is finite, and define
$S^{\mathcal{F}}: L^2 \rightarrow C^1$ by
\begin{align}\notag
S^{\mathcal{F}}F=
\sum_{(j,k) \in {\cal F}} \mu(E_{j,k}) \langle  F, \varphi_{j,k}  \rangle  \psi_{j,k}
\end{align}
By (b), $S^{\mathcal{F}}: L^2 \rightarrow L^2$ is bounded, with norm
$\|S^{\mathcal{F}}\| \leq C_2$
{\em for all } $N$.  (Indeed, in the formula for $S$,
we have just replaced the $\varphi_{j,k}$ by the
zero function if $j,k \notin {\cal F}$, and the zero function is surely in
all ${\mathcal M}_{x_{j,k},a^j}$.)  Since (c) is true for $F \in C^1$, which
is dense in $L^2$, it now follows for all $F \in L^2$.
\end{proof}
\begin{proof}[Proof of (d)]

This follows at once from (c), since a series of complex numbers conveges
absolutely if and only if it converges unconditionally.

This completes the proof of Theorem \ref{sumopthm}.\end{proof}

%
We are now almost ready to justify the transition from (\ref{kerfaj}) to (\ref{kerfajap}).
In order to do so, we need to choose the $E_{j,k}$ rather carefully.  We will need to
utilize the following simple fact:

\begin{itemize}
\item Say $t > 0$.  Then there exists a finite covering of ${\bf M}$ by
disjoint measurable sets $E_1,\ldots,E_N$, such that whenever $1 \leq k \leq N$, there
is a $y_k \in {\bf M}$ with $B(y_k,t) \subseteq E_k \subseteq B(y_k,2t)$.  Thus, by
(\ref{ballsn}), there is a constant $c_0^{\prime}$, depending only on ${\bf M}$, such that
$\mu(E_k) \geq c_0^{\prime}(2t)^n$, provided $t < \delta$.\\
\ \\
To see this, choose a maximal collection of disjoint balls $\{B(y_k,t): 1 \leq k \leq N\}$,
set $B_k = B(y_k,t)$, and set $B_k^{\prime} = B(y_k,2t)$.  Since the collection
is maximal, each point $y \in {\bf M}$ is at distance less than $t$ from some point in
$\cup_{k=1}^N B_k$, so the distance from $y$ to some $y_k$ is less than $2t$.  Therefore
the $B_k^{\prime}$ cover ${\bf M}$, and we need only set $E_1 =
B_1^{\prime}\setminus (B_2 \cup \ldots \cup B_N)$, and recursively, for $2 \leq k \leq N$,
$E_k = B_k^{\prime} \setminus (E_1 \cup \ldots E_{k-1} \cup B_{k+1} \cup \ldots B_N)$.
\vspace{.3cm}

We conclude:
\item
Say $c_0^{\prime}$ is as in the last bullet point, and say $0 < b < 1$.  Then,
for each $j \in \ZZ$ with $ba^j < 2\delta$, we
can write ${\bf M}$ as a finite disjoint union of measurable sets
$\{E_{j,k}\}$, where the diameter of each $E_{j,k}$ is less than or equal to $ba^j$, and where, for each
$j$ with $ba^j < 2\delta$, $\mu(E_{j,k}) \geq c_0(ba^j)^n$. \\
\ \\
This follows from the last bullet point, with $t = ba^j/2$.\\
\end{itemize}
The following theorem now justifies the transition from (\ref{kerfaj}) to (\ref{kerfajap}), and
in particular implies Theorem \ref{framain0}.
\begin{theorem}
\label{framain}
$(a)$ Fix $a >1$, and say $c_0 , \delta_0 > 0$.  Suppose $f \in {\mathcal S}(\RR^+)$,
and $f(0) = 0$.
Then there exists a constant $C_0 > 0$ $($depending only on ${\bf M}, f, a, c_0$ and $\delta_0$$)$
as follows:\\
Let ${\mathcal J} \subseteq \ZZ$ be such that for some $N$, $j \in {\mathcal J}$
whenever $|j| > N$.  $($We will often take ${\mathcal J} = \ZZ$.$)$
For $t > 0$, let $K_t$ be the kernel of $f(t^2\Delta)$.
For $x,y \in {\bf M}$, set
\[ \Phi_{x,t}(y) = \overline{K}_t(x,y). \]
Say $0 < b < 1$.
Suppose that, for each $j \in {\mathcal J}$, we can write ${\bf M}$ as a finite disjoint union of measurable sets
$\{E_{j,k}: 1 \leq k \leq N_j\}$,
where:
\begin{equation}
\label{diamleq}
\mbox{the diameter of each } E_{j,k} \mbox{ is less than or equal to } ba^j,
\end{equation}
and where:
\begin{equation}
\label{measgeq}
\mbox{for each } j \mbox{ with } ba^j < \delta_0,\: \mu(E_{j,k}) \geq c_0(ba^j)^n.
\end{equation}
$($By the last bullet point, we
can surely do this if $c_0 \leq c_0^{\prime}$ and $\delta_0 \leq 2\delta$.$)$
Select $x_{j,k} \in E_{j,k}$ for each $j,k$.
By Lemma $\ref{manmol}$, there is a constant $C$ $($independent of the
choice of $b$ or the $E_{j,k}$$)$, such that $\Phi_{x_{j,k},a^j}\in C{\mathcal M}_{x_{j,k},a^j}$
for all $j,k$.  Thus, if for $1 \leq k \leq N_j$, we set
\[ \varphi_{j,k} = \psi_{j,k} =
\begin{cases}
\Phi_{x_{j,k},a^j} & \mbox{ if}\;\; j \in {\mathcal J},  \cr
0  & \text{otherwise},  \cr
\end{cases}
\]
we may thus form the summation operator $S^{\mathcal J}$ with

\begin{equation}\notag
S^{\mathcal J}F = S_{\{\varphi_{j,k}\},\{\psi_{j,k}\}}F =
\sum_{j}\sum_k \mu(E_{j,k}) \langle  F, \varphi_{j,k}  \rangle  \psi_{j,k}.
\end{equation}
 and Theorem $\ref{sumopthm}$ applies.\\

Let $Q^{\mathcal J} = \sum_{j \in {\mathcal J}} |f|^2(a^{2j} T)$
(strong limit, as guaranteed by $(\ref{gelmay2})$.)
  Then for all $F \in L^2({\bf M})$,

\begin{equation}
\label{qsclose}
\left| \langle  (Q^{\mathcal J}-S^{\mathcal J})F,F  \rangle \right| \leq C_0b  \langle  F,F  \rangle
\end{equation}
$($or, equivalently, since $Q^{\mathcal J}-S^{\mathcal J}$ is self-adjoint,
$\|Q^{\mathcal J}-S^{\mathcal J}\| \leq C_0b$.$)$\\
$(b)$ In $(a)$, take ${\mathcal J} = \ZZ$; set $Q =
Q^{\ZZ}$, $S = S^{\ZZ}$.  Use the notation of $(\ref{gelmay2})$, where now
$T = \Delta$ on $L^2({\bf M})$;
suppose in particular that the Daubechies condition $(\ref{daub})$ holds.
Then,
if $P$ denotes the projection in $L^2({\bf M})$ onto the space of constants, we have
\begin{equation}
\label{snrtgt}
(A_a-C_0b)(I-P) \leq S \leq (B_a + C_0b)(I-P)
\end{equation}
as operators on $L^2({\bf M})$.  Thus, for any $F \in (I-P)L^2({\bf M})$,
\begin{equation}\notag
(A_a-C_0b)\|F\|^2 \leq \sum_{j,k} \mu(E_{j,k})| \langle  F,\varphi_{j,k}  \rangle |^2
\leq (B_a + C_0b)\|F\|^2,
\end{equation}
so that, if $A_a - C_0b > 0$, then
$\left\{ \mu(E_{j,k})^{1/2}\varphi_{j,k}\right\}_{j,k}$ is a frame for $(I-P)L^2({\bf M})$,
with frame bounds $A_a - C_0b$ and $B_a + C_0b$.
\end{theorem}
{\bf Note}
By (\ref{daubest}), $B_a/A_a = 1 + O(|(a-1)^2 (\log|a-1|)|)$; evidently
$(B_a + C_0b)/(A_a - C_0b)$ can be made arbitrarily close to $B_a/A_a$ by
choosing $b$ sufficiently small.  So we have constructed ``nearly tight" frames
for $(I-P)L^2({\bf M})$.
\begin{proof}
  We prove (a).  To simplify the notation, in the proof of (a), $j$ will
always implicitly be restricted to lie in ${\mathcal J}$, unless otherwise explicitly stated.

Since $\delta_0$ only occurs in (\ref{measgeq}), we may assume $\delta_0 \leq \delta$; for otherwise, we may
replace $\delta_0$ by $\delta$, and (\ref{measgeq}) still holds.

Since $Q^{\mathcal J}$ and $S^{\mathcal J}$ are bounded operators, we need only show (\ref{qsclose})
for the dense subspace $C^1({\bf M})$.  For $b > 0$, we let $\Omega_b =
\log_a (\delta_0/b)$, so that $j < \Omega_b$ is equivalent to
$ba^j < \delta_0$.
Observe that
\begin{align}\notag
\left| \langle  (Q^{\mathcal J}-S^{\mathcal J})F,F  \rangle \right|
& = \left|\sum_j \int_{\bf M} \left| \langle  \Phi_{x,a^j},F  \rangle  \right|^2 d\mu(x)
- \sum_j \sum_k \mu(E_{j,k})\left| \langle  \Phi_{x_{j,k},a^j},F  \rangle  \right|^2\right|\\\notag
& \leq I + II,
\end{align}
where
\[ I = \left| \sum_{j < \Omega_b}\left[\int_{\bf M} \left| \langle  \Phi_{x,a^j},F  \rangle  \right|^2 d\mu(x) -
\sum_k \mu(E_{j,k})\left| \langle  \Phi_{x_{j,k},a^j},F  \rangle  \right|^2\right] \right|, \]
and
\[ II = \sum_{j \geq \Omega_b} \int_{\bf M} | \langle  \Phi_{x,a^j},F  \rangle  |^2 d\mu(x)
+ \sum_{j \geq \Omega_b} \sum_k \mu(E_{j,k})| \langle  \Phi_{x_{j,k},a^j},F  \rangle  |^2. \]

For $II$, we need only recall that, for any $L > 0$,
\begin{equation}
\label{ktinf}
\lim_{t \rightarrow \infty} t^L  K_t = 0 \mbox{ in } C^{\infty}({\bf M} \times {\bf M}).
\end{equation}
(This follows at once from the eigenfunction expansion of $K_t$; see Section 4 of \cite{gmcw},
especially (\ref{kerexp}) of that article, and the comments that directly follow it.)
In particular, there exists $C > 0$ such that, for any $x$ and any $t > 1$,
$\|\Phi_{x,t}\|_2^2 \leq C/t$.  Accordingly,

\[ II\leq C \sum_{j \geq \Omega_b} a^{-j}\mu({\bf M})\|F\|_2^2 \leq Ca^{-\Omega_b}\|F\|_2^2
= Cb\|F\|_2^2. \]
\ \\
Thus we may focus our attention on $I$.  We have
\begin{align}\notag
I & = \left| \sum_{j < \Omega_b}\sum_k \left[\int_{E_{j,k}} \left| \langle  \Phi_{x,a^j},F  \rangle  \right|^2 d\mu(x) -
\sum_k \mu(E_{j,k})\left| \langle  \Phi_{x_{j,k},a^j},F  \rangle  \right|^2\right] \right|,\\
& = \left| \sum_{j < \Omega_b}\sum_k \int_{E_{j,k}}
\left[\left| \langle  \Phi_{x,a^j},F  \rangle  \right|^2 - \left| \langle  \Phi_{x_{j,k},a^j},F  \rangle  \right|^2\right]d\mu(x) \right|.
\label{iprep}
\end{align}
For each $(j,k)$ with $j < \Omega_b$ and
$1 \leq k \leq N_j$, we select $i_{j,k}$
such that $B(x_{j,k},3\delta) \subseteq U_{i_{j,k}}$.  (Since then
$\diam E_{j,k} \leq ba^j <  \delta$, we then also have that
$E_{j,k} \subseteq B(x_{j,k},\delta) \subseteq U_{i_{j,k}}$.)  For each $i$
with $1 \leq i \leq N$, we let ${\mathcal S}_i$ denote the set of
$(j,k)$, with $1 \leq k \leq N_j$, such that $i_{j,k} = i$.
The summation in (\ref{iprep}) may then be written as $\sum_{i=1}^N I_i$, where

\begin{equation}
\label{iiprep}
I_i = \sum_{(j,k) \in {\mathcal S}_i} \int_{E_{j,k}}
\left[\left| \langle  \Phi_{x,a^j},F  \rangle  \right|^2 - \left| \langle  \Phi_{x_{j,k},a^j},F  \rangle  \right|^2\right] d\mu(x).
\end{equation}
For (a), we only need show that each $|I_i| \leq C\|F\|^2_2$.

Let us then fix $i$ and use local coordinates on $U_i$.  In these coordinates,
$U_i$ is a ball in $\RR^n$.  On $U_i$, we may write $d\mu(x)
= h(x)dx$, where $h$ is smooth, positive and bounded above.  By (\ref{rhoeuccmp}),
if $(j,k) \in {\mathcal S}_i$, then $E_{j,k} \subseteq B(x_{j,k},ba^j)
\subseteq \left\{y \in \RR^n: |x-x_{j,k}| < ba^j/c_1\right\}$.  In the
integrand in (\ref{iiprep}), we may therefore
write $x=x_{j,k} + (ba^j/c_1)w$ for some $w$ in the open unit ball in $\RR^n$,
which we denote by ${\mathcal B}$.   Changing variables in (\ref{iiprep}), we
now see that

\begin{align}\notag
I_i & = \sum_{(j,k) \in {\mathcal S}_i} \int_{B(x_{j,k},ba^j)}
\left[\left| \langle  \Phi_{x,a^j},F  \rangle  \right|^2 - \left| \langle  \Phi_{x_{j,k},a^j},F  \rangle  \right|^2\right]
\left[\chi_{E_{j,k}}h\right](x)dx \\
&= \sum_{(j,k) \in {\mathcal S}_i} \left(ba^j/c_1\right)^n \int_{\mathcal B}
\left[\left| \langle  \Phi_{x_{j,k} + (ba^j/c_1)w,a^j},F  \rangle \right|^2
- \left| \langle  \Phi_{x_{j,k},a^j},F  \rangle  \right|^2\right] H_{j,k}(w) dw, \label{iifst}
\end{align}

where
\begin{equation}\notag
H_{j,k}(w) = [\chi_{E_{j,k}}h](x_{j,k} + (ba^j/c_1)w).
\end{equation}
(This is interpreted as zero if $x_{j,k} + (ba^j/c_1)w \notin E_{j,k}$.
Moreover the integrand in (\ref{iifst}) is interpreted as zero at any point $w$
where $H_{j,k}(w) = 0$.)
Thus

\begin{equation}
\label{iiscd}
 I_i = \sum_{(j,k) \in {\mathcal S}_i} \mu(E_{j,k})
\int_{\mathcal B}
\left[\left| \langle  \Phi_{x_{j,k} + (ba^j/c_1)w,a^j},F  \rangle \right|^2
- \left| \langle  \Phi_{x_{j,k},a^j},F  \rangle  \right|^2\right] G_{j,k}(w) dw,
\end{equation}
where
\[ G_{j,k}(w) = \frac{(ba^j)^n}{c_1^n \mu(E_{j,k})} H_{j,k}(w) \]
Note that, by (\ref{measgeq}), there is a constant $C$ such that
$0 \leq G_{j,k}(w) \leq C$ for all $(j,k) \in {\mathcal S}_i$ and all $w \in
{\mathcal B}$. (Again, the integrand in (\ref{iiscd}) is interpreted as zero at any point
$w$ where $G_{j,k}(w) = 0$.)

Applying the fundamental theorem of calculus, and recalling that $U_i$ (being
a ball) is convex, we now see that

\begin{equation}
\label{iibtr}
I_i = \sum_{(j,k) \in {\mathcal S}_i} \mu(E_{j,k}) \int_{\mathcal B}
\int_0^1 \frac{\partial}{\partial s}\left| \langle  \Phi_{x_{j,k} + (ba^j/c_1)sw,a^j},F  \rangle \right|^2 ds\: G_{j,k}(w) dw.
\end{equation}
For $(j,k) \in {\mathcal S}_i$, $w \in {\mathcal B}$ and $0 \leq s \leq 1$,
$y \in {\bf M}$, let us set
\begin{align}\notag
\varphi_{j,k}^{w,s}(y) & = \sqrt{G_{j,k}(w)}\Phi_{x_{j,k} + (ba^j/c_1)sw,a^j}(y) \\\notag
&= \sqrt{G_{j,k}(w)}\overline{K}_{a^j}(x_{j,k} + (ba^j/c_1)sw,y)
\end{align}
(interpreted at zero if $G_{j,k}(w) = 0$.)
Also note that, for $y \in {\bf M}$,
\begin{align}\notag
\sqrt{G_{j,k}(w)}\frac{\partial}{\partial s} \Phi_{x_{j,k} + (ba^j/c_1)sw,a^j}(y)
& = \sqrt{G_{j,k}(w)}\frac{\partial}{\partial s} \overline{K}_{a^j}(x_{j,k} + (ba^j/c_1)sw,y)\\\notag
& = b\psi_{j,k}^{w,s}(y)
\end{align}
where
\[ \psi_{j,k}^{w,s}(y)= \left[\sqrt{G_{j,k}(w)}/c_1\right]
\sum_{m=1}^n w_m \left[a^j\frac {\partial \overline{K}_{a^j}}{\partial x_m}
(x_{j,k} + (ba^j/c_1)sw,y)\right]. \]
Now, if $w$ is such that $G_{j,k}(w) \neq 0$, then $x_{j,k} + (ba^j/c_1)w \in E_{j,k}
\subseteq U_i$, so $x_{j,k} + (ba^j/c_1)sw \in U_i$ for all $0 \leq s \leq 1$.
Since $|[x_{j,k} + (ba^j/c_1)sw] - x_{j,k}| \leq a^j/c_1$, by
(\ref{rhoeuccmp2}), we have that $d(x_{j,k} + (ba^j/c_1)sw, x_{j,k}) \leq (c_2/c_1)a^j$.
By (\ref{diagest}) and (\ref{nrbyok}), there is consequently a $C > 0$ such that
\[ \varphi_{j,k}^{w,s}, \psi_{j,k}^{w,s} \in C{\mathcal M}_{x_{j,k},a^j} \]
for all $(j,k) \in {\mathcal S}_i$, $w \in {\mathcal B}$ and $0 \leq s \leq 1$.
Now let us set $\varphi_{j,k}^{w,s}, \psi_{j,k}^{w,s} \equiv 0$ whenever $(j,k)
\notin {\mathcal S}_i$ (here $j$ runs from $-\infty$ to
$\infty$, and $k$ runs from $1$ to $N_j$). We now find from (\ref{iibtr}) that
\begin{align}\notag
I_i & = b\sum_{(j,k) \in {\mathcal S}_i} \mu(E_{j,k}) \int_{\mathcal B}
\int_0^1 \left[ \langle  \psi_{j,k}^{w,s},F  \rangle   \langle  F, \varphi_{j,k}^{w,s} \rangle  +
 \langle  \varphi_{j,k}^{w,s},F  \rangle   \langle  F,\psi_{j,k}^{w,s} \rangle \right] dsdw\\\notag
& = b \int_{\mathcal B}
\int_0^1 \left[ \langle  S_{\{\varphi_{j,k}^{w,s}\},\{\psi_{j,k}^{w,s}\}}F,F  \rangle
+  \langle  S_{\{\psi_{j,k}^{w,s}\},\{\varphi_{j,k}^{w,s}\}}F,F  \rangle \right] dsdw\\\notag
& \leq Cb\|F\|^2
\end{align}
by Theorem \ref{sumopthm}.  (The interchange of order of summation and integration
is justified by the dominated convergence theorem and the second sentence of
Theorem \ref{sumopthm} (a).)   This proves (a).

To prove (b), we need only show (\ref{snrtgt}).  But from
(\ref{qsclose}) and (\ref{gelmay2}), we have that, if $F = (I-P)F \in L^2({\bf M})$, then
\[ (A_a - C_0b)\|F\|^2 \leq
 \langle  QF, F  \rangle  - C_0b\|F\|^2 \leq
 \langle  SF, F  \rangle  \leq  \langle  QF, F  \rangle  + C_0b\|F\|^2
\leq (B_a + C_0b)\|F\|^2. \]
If, on the other hand, $F \in L^2({\bf M})$ is general, we have $SF = S(I-P)F$,
since all $\varphi_{j,k}$ have integral zero.  Since
$S$ is self-adjoint, $ \langle  SF, F  \rangle  =  \langle  S(I-P)F, (I-P)F  \rangle $, so in
general
\[ (A_a - C_0b)\|(I-P)F\|^2 \leq  \langle  SF, F  \rangle  \leq
(B_a + C_0b)\|(I-P)F\|^2, \]
as desired.
\end{proof}

In Theorem \ref{framain} (b) $\{\mu(E_{j,k})^{1/2}\varphi_{j,k}\}$ is a frame.
This is, of course, an infinite set, and so for practical purposes, we must explain which
terms in the summation

\begin{equation}
\label{srepr}
SF = \sum_j \sum_k \mu(E_{j,k})  \langle  F,\varphi_{j,k}  \rangle  \varphi_{j,k}
\end{equation}
are so small that they can be neglected.  As we stated in the introduction, our plan is
first to do a frequency analysis. (Assuming, roughly, that $F$ is well-localized in
frequency, we describe which $j$ can be neglected).  This will reduce us to a finite sum.
Then we will do a spatial analysis. (Assuming, roughly, that $F$ is well-localized in
space, we describe which $k$ can be neglected for each of our now finite set of $j$).

\subsection{Frequency Analysis}

We begin with some motivation for the result we seek.

Even without using frequency analysis, it is not difficult to see that the terms with
$j$ large and positive will contribute little to the sum (\ref{srepr}).
Indeed, for any $L > 0$,
there exists $c_L > 0$ such that if $F \in L^2$, then
\[ C_{j,L} := | \langle  F, \varphi_{j,k} \rangle | \leq \|F\|_2\|\varphi_{j,k}\|_2 \leq
c_L a^{-Lj}\|F\|_2, \]
by (\ref{ktinf}).
Following (\ref{sumsetup}) through (\ref{ccjrhere}) with $C_{j,L}$ in place of $C_{j,R}$,
we see that for each $j$,
\begin{equation}
\label{adupok}
\sum_k \mu(E_{j,k}) \left|  \langle  F,  \varphi_{j,k} \rangle \right| \:\left| \varphi_{j,k}(y)\right|
\leq   C  C_{j,L}.
\end{equation}

Taking  the sum over $j > N$, say, we have:
\begin{equation}
\label{expdec}
\sum_{j > N} \sum_k   \mu(E_{j,k}) \left|  \langle  F,  \varphi_{j,k} \rangle  \right|\:\left|
\varphi_{j,k}(y) \right|  \leq  C   \sum_{j > N} C_{j,L}
\leq  C_L \|F\|_2 a^{-LN},
\end{equation}
which goes to zero rapidly as $N \rightarrow \infty$.  Here we need no information
about $F$ except its $L^2$ norm.\\
\ \\
On the other hand, it is not hard to see that if $F$ has a little smoothness, then the
terms with $-j$ large and positive will contribute little to the sum (\ref{srepr}).
If, for instance,
$F$ is H\"older continuous of exponent $\alpha \in (0,1)$, the arguments of
(\ref{goholder}) through (\ref{endholder}) show that $| \langle  F, \varphi_{j,k} \rangle |
\leq C a^{j\alpha}$, so if we argue as in (\ref{adupok}) and (\ref{expdec}) above,
we see that
\[ \sum_{j < -M} \sum_k  \mu(E_{j,k}) \left|  \langle  F,  \varphi_{j,k} \rangle  \right|\: \left|
\varphi_{j,k}(y) \right| \leq C \sum_{j < -M} a^{j\alpha} \leq C a^{-M\alpha},  \]
which also goes to zero rapidly as $M \rightarrow \infty$.\\
\ \\
We leave the detailed study (through frames) of smoothness spaces to a later article (\cite{gm3}).
Here we instead look at the terms, with $|j|$ large, through frequency analysis.

Specifically, say $F = \sum a_l u_l$ is the expansion of $F$ with respect to the
orthonormal basis of eigenfunctions $\{u_l\}$.  We expect that some knowledge of the decay
of the $\{a_l\}$ should imply something about the decay of the terms in (\ref{srepr})
as $j \rightarrow -\infty$.  After all, for instance, if the $a_l$ decay quickly enough,
then $F$ is $C^1$.

Our next result gives us such information, in considerable generality.  Moreover,
the fact that the terms in (\ref{srepr}) become negligible as $j \rightarrow
\infty$, will be seen to be a consequence of the fact that, if $m \geq 1$, then the
eigenvalues of $\Delta$ satisfy $\lambda_m \geq \lambda_1 > 0$.  The proof of the result
will depend on Lemma \ref{gelmayquant} of our earlier article \cite{gmcw}.

\begin{theorem}
\label{freqthm}
In the situation of Theorem $\ref{framain}$  $(b)$,
suppose in fact that $f$ has the form $f(s) = s^l f_0(s)$ for some
integer $l \geq 1$ and $f_0 \in {\cal S}(\RR^+)$, $(f_0 \not\equiv 0)$.
Suppose $J \geq 1$ is an integer, and let $M_J = \max_{r > 0}|r^J f(r)|$.
For any $L > 0$,
we consider the spectral projectors $P_{[0,L]}$ $($so that, if $F = \sum a_m u_m
\in L^2$, then $P_{[0,L]} F = \sum_{\lambda_m \leq L} a_m u_m)$.
Then if $F \in L^2$,
and any $M,N \geq 0$, we have
\begin{equation}
\label{freqest}
\|SF - \sum_{j=-M}^N \sum_k \mu(E_{j,k})  \langle  F,\varphi_{j,k}  \rangle  \varphi_{j,k}\|_2
\leq ( C_0b + \frac{c^{\prime}_L}{a^{4Ml}} + \frac{C_J^{\prime}}{a^{4NJ}})\|F\|_2 +
2B_a \|(I - P_{[0,L]})F\|_2,
\end{equation}
where $c^{\prime}_L = (L^{2l}\|f_0\|^2_{\infty})/(a^{4l}-1)$, and
$C^{\prime} = M_J^2/[(a^{4J}-1)\lambda_1^{2J}]$.\\
\end{theorem}
\begin{proof} First note, that in proving this result, we may assume $F = (I-P)F$;
otherwise, replace $F$ by $(I-P)F$ (this will not affect the left side of
(\ref{freqest}), and will not increase the right side).  It then follows that
$(I - P_{[0,L]})F = (I - P_{[\lambda_1,L]})F$.

Let ${\mathcal J} = \{j \in \ZZ: j < -M \mbox{ or } j > N\}$.  Then
the left side of (\ref{freqest}) is simply $\|S^{\mathcal J}F\|_2$, in the notation
of Theorem \ref{framain} (a).  By that result, the left side of
(\ref{freqest}) is less than or equal to $C_0 b\|F\|_2 +\|Q^{\mathcal J}F\|_2$.
But in the notation of Lemma \ref{gelmayquant} of \cite{gmcw},
 $Q^{\mathcal J} = h(\Delta) - h_{M,N}(\Delta)$.
Thus the theorem now follows at once from Lemma \ref{gelmayquant} of \cite{gmcw} (with $\eta$
in that lemma equalling $\lambda_1$).  This completes the proof.\end{proof}

Thus, say one wants to compute $SF$ to a certain precision.  One calculates
the finite sum
\begin{equation}
\label{MNappS}
\sum_{j=-M}^N \sum_k \mu(E_{j,k})  \langle  F,\varphi_{j,k}  \rangle  \varphi_{j,k}
\end{equation}
for $M, N$ sufficiently large;
how large must one take them to be?  One first chooses $b, J, L, N$ so that the
first, third and fourth terms on the right side of (\ref{freqest}) are
very small.  Then one chooses $M$ to make the second term on the right side of
(\ref{freqest}) very small as well.

We emphasize that $C_J^{\prime}$ in (\ref{freqest}) does not depend on $L$, so that
the negligibility of the terms as $j \rightarrow \infty$ depends only on
$\|F\|_2$, as we observed above.  $C^{\prime}$ of course depends on $f, a$ and also on
the manifold ${\bf M}$ (more specifically, on the least positive eigenvalue $\lambda_1$).

\subsection{Spatial Analysis}
The summation in (\ref{MNappS}) is a finite sum, but,
as we have motivated in the Introduction, if we only want to know this
summation to within a certain precision, then it may well be possible to further reduce
the number of terms that need to be considered, provided
$F$ is well-localized in space.

In order to carry out the analysis, it is convenient to first make a general remark.
 Let ${\cal I} \subseteq \{(j,k): 1 \leq k \leq N_j\}$ be arbitrary.    Let
\begin{equation}
\label{scalidf}
S_{\cal I}F =
\sum_{(j,k) \in \cal I} \mu(E_{j,k})  \langle  F,\varphi_{j,k}  \rangle  \varphi_{j,k}.
\end{equation}
Then $S_{\cal I}$ is a positive operator on $L^2({\bf M})$; let $\sqrt{S_{\cal I}}$
be its positive square root.  For any $F \in L^2$,
\[ \|\sqrt{S_{\cal I}} F\|_2^2 =  \langle  S_{\cal I} F, F  \rangle
= \sum_{j,k \in {\cal I}}\mu(E_{j,k})| \langle  F,\varphi_{j,k} \rangle |^2
\leq B \|F\|^2_2, \]
where $B$ is an upper frame bound for the frame $\{\mu(E_{j,k})^{1/2}\varphi_{j,k}\}$.
Accordingly, $\|\sqrt{S_{\cal I}}\| \leq \sqrt{B}$, so
\begin{equation}
\label{nrok2}
\|S_{\cal I}\| \leq B.
\end{equation}
Also, for any $F \in L^2$,
$\|S_{\cal I}F\|_2 \leq \|\sqrt{S_{\cal I}}\| \|\sqrt{S_{\cal I}} F\|_2$, so
\begin{equation}
\label{nrok}
\|S_{\cal I}F\|_2^2 \leq B \langle  S_{\cal I}F, F  \rangle .
\end{equation}
Thus $S_{\cal I}F$ will have small $L^2$ norm provided $ \langle  S_{\cal I}F, F  \rangle $
is sufficiently small.  Accordingly, the terms in (\ref{MNappS}) corresponding to
$(j,k) \in {\cal I}$ can be neglected, provided $ \langle  S_{\cal I}F, F  \rangle $
is sufficiently small.\\
\ \\
We may now prove out result on spatial analysis:
\begin{theorem}
\label{spacest}
In the situation of Theorem $\ref{framain}$ $(b)$, say $I > 0$.
Then there exists a constant $C_1 > 0$
$($depending only on $I, f$ and ${\bf M})$, as follows.

Let $B$ be an upper frame bound for the frame $\left\{ \mu(E_{j,k})^{1/2}\varphi_{j,k}\right\}$.
Say $M,N \geq 0$, $x_0 \in {\bf M}$, and $R > 0$.
Let $\chi$ be the characteristic function of a set $\Gamma \subseteq {\bf M}$.
For each $j$ with $-M \leq j \leq N$, let $c_j > 0$ be a constant, and let
\[ {\mathcal I} = \left\{(j,k):\; -M \leq j \leq N,~ 1 \leq k \leq N_j, ~d(x_{j,k},\Gamma)
\leq (c_j + 1) a^j
\right\} \]
Then
\begin{align}\notag
 &\left\|\sum_{j=-M}^N \sum_k \mu(E_{j,k})  \langle  F,\varphi_{j,k}  \rangle  \varphi_{j,k}
- \sum_{j,k \in {\mathcal I}}\mu(E_{j,k})  \langle  F,\varphi_{j,k}  \rangle  \varphi_{j,k}\right\|_2
  \\
 \leq&  C_1B^{1/2}\left[\mu(\Gamma) \sum_{j=-M}^N a^{-jn}c_j^{n-2I}\right]^{1/2} \|\chi F\|_2
+ B\|(1-\chi)F\|_2 \label{spacesteq}
\end{align}
for all $F \in L^2({\bf M})$.
\end{theorem}
\begin{proof} Let ${\mathcal I}' =
\left\{(j,k):\; -M \leq j \leq N, ~1 \leq k \leq N_j, ~d(x_{j,k},\Gamma) > (c_j + 1) a^j \right\}$;
then, in the
notation of (\ref{scalidf}), the left side of (\ref{spacesteq}) is $\|S_{{\mathcal I}'}F\|_2$.
Surely
\begin{equation}
\label{omcest}
\|S_{{\mathcal I}'}F\|_2 \leq \|S_{{\mathcal I}'}(\chi F)\|_2 + \|S_{{\mathcal I}'}([1-\chi]F)\|_2
\leq \|S_{{\mathcal I}'}(\chi F)\|_2 + B\|(1-\chi)F\|_2
\end{equation}
by (\ref{nrok2}).  On the other hand, if $(j,k) \in {\mathcal I}'$, then
\begin{eqnarray*}
\left| \langle  \chi F, \varphi_{j,k}  \rangle \right|^2 & \leq &
\left[\int_{\Gamma} \left|(\chi F)(y)\right| \left|\varphi_{j,k}(y)\right| d\mu(y)\right]^2\\
& \leq & \left[\int_{\Gamma} \left|\varphi_{j,k}(y)\right|^2 d\mu(y)\right] \|\chi F\|_2^2\\
& \leq & C_I^2 a^{-2jn}\left[\int_{\Gamma} \left(1 + d(x_{j,k},y)/a^j\right)^{-2I} d\mu(y)\right] \|\chi F\|_2^2.
\end{eqnarray*}
For each $j$ with $-M \leq j \leq N$,
let $E_j = \underset{{\{k:\;(j,k) \in {\mathcal I}'\}}}{\cup}E_{j,k}$.  If $x \in E_j$, then $x \in E_{j,k}$
for some $(j,k)$, so for each $y \in \Gamma$, $d(x,y) \geq d(x_{j,k},y) - d(x_{j,k},x)
> c_j a^j$.  Thus, by (\ref{alcmpN}) and then (\ref{ptestm1}), we see that

\begin{eqnarray*}
 \langle  S_{{\mathcal I}'}(\chi F),(\chi F)  \rangle
  & = & \sum_{(j,k) \in {\mathcal I}'} \mu(E_{j,k})| \langle  \chi F, \varphi_{j,k}  \rangle |^2 \\
& \leq & C_I^2 \int_{\Gamma}
\sum_{(j,k) \in {\mathcal I}'} a^{-2jn}
\mu(E_{j,k})\left(1 + d(x_{j,k},y)/a^j\right)^{-2I} d\mu(y) ~\|\chi F\|_2^2\\
& \leq & C_I^2 2^{2I}\sum_{j=-M}^N a^{-2jn}\int_{\Gamma} \int_{E_j}
\left(1 + d(x,y)/a^j\right)^{-2I} d\mu(x) d\mu(y)~ \|\chi F\|_2^2\\
& \leq & C_I^2 2^{2I}\sum_{j=-M}^N a^{-2jn}\int_{\Gamma} \int_{\left\{x:\: d(x,y) > c_j a^j\right\}}
\left(d(x,y)/a^j\right)^{-2I} d\mu(x) d\mu(y) ~\|\chi F\|_2^2\\
& \leq & C_I' C_I^2 2^{2I}\sum_{j=-M}^N a^{-jn}c_j^{n-2I}
\int_{\Gamma} d\mu(y)~ \|\chi F\|_2^2\\
& \leq & C_I'C_I^2 2^{2I}\sum_{j=-M}^N a^{-jn}c_j^{n-2I}\mu(\Gamma) ~\|\chi F\|_2^2\\
\end{eqnarray*}
Accordingly, by (\ref{nrok}) we find
\begin{equation}\notag
\|S_{{\mathcal I}'}(\chi F)\|_2 \leq
C_1 (B \mu(\Gamma))^{1/2}\left[\sum_{j=-M}^N a^{-jn}c_j^{n-2I}\right]^{1/2} \|\chi F\|_2.
\end{equation}
with $C_1$ depending only on $I, f$ and ${\bf M}$.
If we combine this with (\ref{omcest}), we see that the proof is complete.\end{proof}

In all, say one wants to compute $SF$ to a certain precision.  By using frequency
analysis, one can reduce the problem to computing
$\sum_{j=-M}^N \sum_k \mu(E_{j,k})  \langle  F,\varphi_{j,k}  \rangle  \varphi_{j,k}$
to a certain precision.  If $F$ is well-localized, by  Theorem \ref{spacest}, one
can further reduce the problem to computing a sum
$\sum_{j,k \in {\mathcal I}}\mu(E_{j,k})  \langle  F,\varphi_{j,k}  \rangle  \varphi_{j,k}$, by
first choosing the set $\Gamma$ so that $B\|(I-\chi)F\|_2$ is small, and then
choosing the numbers $c_j$ so that $\sum_{j=-M}^N a^{-jn}c_j^{n-2I}$ is sufficiently
small.  This will entail making the numbers $c_j$ sufficiently large.  The needed
computations will increase as $j$ gets smaller, since each $a^{-jn}c_j^{n-2I}$
needs to be small (forcing larger requirements on $c_j$ as $j$ gets smaller),
and since the number of $k$ with $(j,k) \in {\mathcal I}$ also
increases as $j$ gets smaller.  Thus it is best to use frequency analysis to make
$M$ as small as possible.  This will be most practical if $F$ has a little smoothness.

\section{Needlets and Mexican Needlets on the Sphere}

In order to implement our frames, one of course would need to calculate approximately the frame elements
\begin{equation}\notag
\phi_{j,k} = \mu(E_{j,k})^{1/2}~\overline{K_{a^j}}(x_{j,k},x),
\end{equation}
for some collection of $E_{j,k}$ satisfying (\ref{diamleq}), (\ref{measgeq}) (for any given $b \in (0,1)$)
and for some choice of $x_{j,k} \in E_{j,k}$.  In Section 6 of \cite{gmcw}, we focused on the cases
where ${\bf M}$ is the torus $\TT^2$ or the sphere $S^2$.  In these cases it is not difficult
to explicitly choose suitable $E_{j,k}$.  Moreover, in Section 6 of \cite{gmcw}, we gave explicit
approximate formulas for $K_t(x,y)$, if $f(s) = se^{-s}$ (the Mexican hat case), for the torus and
for the sphere, which therefore should make implementation feasible.
In each case ($\TT^2$ or $S^2$), we gave two approximate formulas for $K_t(x,y)$, one which converges
quickly for large $t$, and one which converges quickly for small $t$.  In this section we focus on the
sphere $S^2$, where related work has been done by Narcowich, Petrushev and Ward (\cite{narc1}, \cite{narc2}).
We shall compare our methods with theirs, in detail.

Specifically, on $S^2$, in \cite{gmcw} we noted that $K_t(Tx,Ty) = K_t(x,y)$ for any orthogonal
transformation $T$.  From this we concluded that $K_t(x,y)$ was actually a function of $x \cdot y$,
say
\begin{equation}\notag
h_t(x \cdot y) = K_t(x,y).
\end{equation}
From the eigenfunction expansion of $K_t$, we found that, for $-1 \leq y_1 \leq 1$,
if $y = (y_1,\ldots,y_3) \in S^2$,
\begin{equation}
\label{kersph3}
4\pi h_t(y_1) = K_t({\bf N},y) = \sum_{l=0}^{\infty} (2l+1)f(t^2l(l+1)) P_l^{1/2}(y_1).
\end{equation}

Here $P_l^{1/2}$ is a Gegenbauer polynomial.  This series converges quickly when $t$ is large.  When $t$ is small,
and $f(s) = se^{-s}$, we obtained the approximation
\begin{equation}
\label{htapp}
4\pi h_t(\cos \theta) \sim \frac{e^{-\theta^2/4t^2}}{t^2}[(1-\frac{\theta^2}{4t^2})p(t,\theta)-t^2q(t,\theta)],
\end{equation}
where
\begin{equation}\notag
p(t,\theta)=1+\frac{t^2}{3}+\frac{t^4}{15}+\frac{4t^6}{315}+\frac{t^8}{315}+
\frac{\theta^2}{4}(\frac{1}{3}+\frac{2t^2}{15}+\frac{4t^4}{105}+\frac{4t^6}{315})
\end{equation}
and
\begin{equation}\notag
q(t,\theta)= \frac{1}{3}+\frac{2t^2}{15}+\frac{4t^4}{105}+\frac{4t^6}{315}+
\frac{\theta^2}{4}(\frac{2}{15}+\frac{8t^2}{105}+\frac{4t^4}{105})
\end{equation}
Maple says that when $t=.1$, the error in the approximation (\ref{htapp}) is never more than $9.5 \times 10^{-4}$
for any $\theta \in [-\pi,\pi]$, even though both sides have a maximum of about 100.
(To obtain rigorous bounds on the error is research in progress, which we expect to complete soon.)
One can derive similar types of approximations if instead $f(s) = s^r e^{-s}$ for any integer $r \geq 1$,
by applying repeated $t$ derivatives to the above formulas.  (See the remark at the end of Section 6
of \cite{gmcw}.) If in (\ref{htapp})
we approximate $p \sim 1$ and $q \sim 0$, we would obtain the formula for the usual Mexican hat wavelet
on the real line, as a function of $\theta$.  This is to be expected, since on $\RR^n$,
the Mexican hat wavelet is a multiple of $\Delta e^{-\Delta} \delta$, the
Laplacian of a Gaussian, the
function whose Fourier transform is $|\xi|^2 e^{-|\xi|^2}$.  Figure 1 is a graph,
obtained by using Maple, of $4\pi\:h_t(\cos \theta)$,
for $t = 0.1$, with $\theta$ going from $-\pi$ to $\pi$ on the horizontal axis.
  \begin{figure}
  \begin{center}
  \begin{tabular}{c}
  \includegraphics[scale=0.6]{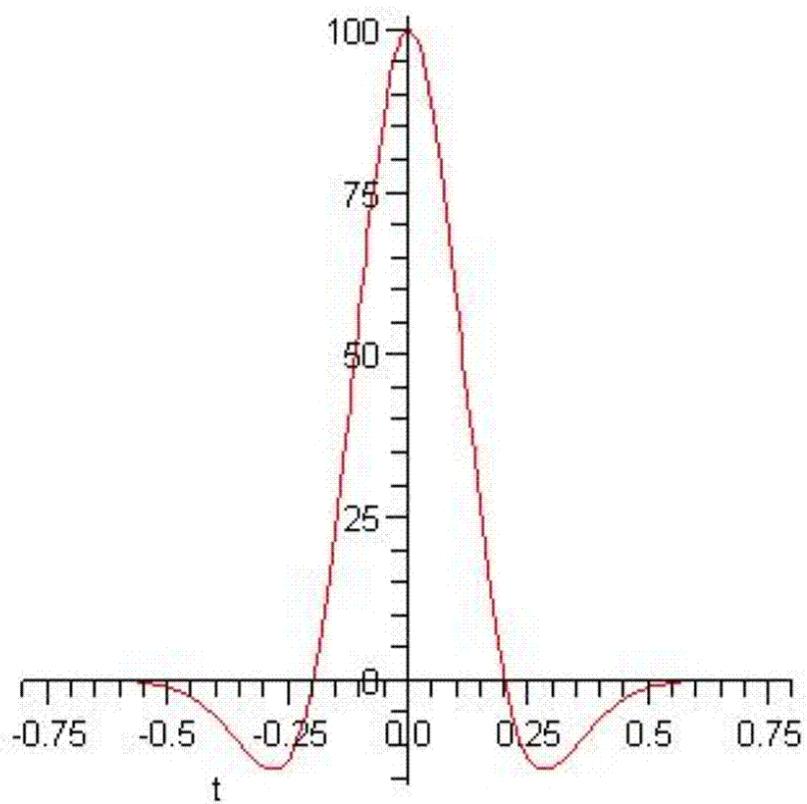}
  \end{tabular}
 \caption{\label{fig 5} $4\pi h_t(\cos \theta)$ on $S^2$ for $t=0.1$}
  \end{center}
  \end{figure}

The related frames of Narcowich, Petrushev and Ward have
been dubbed ``needlets'', and used by astrophysicists to study cosmic microwave background radiation.
(See, for instance, \cite{baldi}, \cite{guil} and the references therein.)  We will therefore call
the frame we obtain when $f(s) = se^{-s}$ (or, more generally, $f(s) = s^r e^{-s}$), by the name
{\em Mexican needlets}.  The formulas above suggest that Mexican needlets have strong Gaussian decay at each
scale, and that (at least for small $r$) they do not oscillate to an extent that would make implementation,
directly on the sphere, prohibitive.
\\

We now give a detailed comparison of needlets and Mexican needlets; there
are advantages and disadvantages in each approach.  In this discussion,
${\cal H}$ will denote $L^2(S^n)$ and $P_l$ will denote the projection onto ${\cal H}_l$, the
space of spherical harmonics of degree $l$.

In \cite{narc1}, \cite{narc2}, the authors considered only ${\bf M} = S^n$.  In place of our $f$ they
consider only smooth $g$ with compact support within $[1/2,2]$
(we shall henceforth call such an $g$ a ``cutoff function'').  They also obtained their frames
from the kernel $K_t$ of an operator, $K_t({\bf N},y)$ being given by a formula similar
to our (\ref{kersph3}), but with $g(tl)$ in place of $f(t^2l(l+n-1))$. (Thus they were not actually considering
functions of $t^2\Delta$, or equivalently functions of $t\sqrt{\Delta}$,
but rather functions of $t{\cal M}$, where ${\cal M}$ is the first-order pseudodifferential
operator $\sum_l lP_l$.  This is a minor distinction, however.)  As we shall explain, the principal
advantages of using cutoff functions is that
the authors are then able to obtain tight frames, and that the frame elements on non-adjacent scales are orthogonal.
The principal disadvantages, as we shall explain, is that there is no reason to expect
explicit formulas on the sphere, Gaussian decay at each scale, or lack of oscillation for needlets.\\

Let us begin our detailed discussion of needlets by explaining how Narcowich, Petrushev and Ward obtain
a tight frame by using cutoff functions:
\begin{itemize}
\item[1.]
 One can easily choose a cutoff function $g$ so that
$\sum_{j=-\infty}^{\infty} |g(2^js)|^2 = 1$ for all $s > 0$ instead of this being only approximately true
(recall our (\ref{daub}) and (\ref{daubest})).
In fact, if $g$ is a cutoff function, there are only two nonvanishing terms in the sum for any given $s > 0$. \item[2.]
  For any $m \geq 0$, let ${\cal P}_m = \sum_{l=1}^m P_l$.
One can choose a finite set of {\em cubature points} $\{x_{m,i}\}$  and
positive numbers $\lambda_{m,i}$ such that for any $F \in {\cal P}_m{\cal H}$,
one has that $\int_{S^n} F = \sum_i \lambda_{m,i} F(x_{m,i})$.  Thus when
evaluating the integral of a function known to be in ${\cal P}_m{\cal H}$, one can use cubature
to evaluate the integral instead of approximating the integral by a Riemann sum (as we did in
passing from (\ref{kerfaj}) to (\ref{kerfajap})).
\end{itemize}

One can construct tight frames out of such $g$, the plan being as follows.  Suppose $\{\nu_l\}_{l=1}^{\infty}$
is a sequence of positive real numbers which increases to $\infty$, with no worse than polynomial growth.
Let ${\cal M}$ be the self-adjoint operator $\sum_{l=1}^{\infty} \nu_l P_l$.
(In \cite{narc1}, \cite{narc2}, the authors take ${\cal M} = \sum lP_l$ as in the last paragraph.
We instead would look at ${\cal M} = \sqrt{\Delta} = \sum \sqrt{l(l+n-1)}P_l$.)
For $t > 0$, let $K_t(x,y)$ be the kernel of $g(t{\cal M})$.
(We would choose $g(s) = f(s^2)$, so that $g(t\sqrt{\Delta}) = f(t^2\Delta)$.)
Since  $\sum_{j=-\infty}^{\infty} |g(2^js)|^2 = 1$,
for all $s > 0$, we have $\sum_{j=-\infty}^{\infty} |g(2^j{\cal M})|^2 = I-P_0$.  Thus, if
$F \in (I-P_0){\cal H}$, we find (analogously to (\ref{aasumba0}), (\ref{aasumba})) that
\begin{equation}
\label{nar1}
\|F\|_2^2 = \langle \sum_{j=-\infty}^{\infty} |g(2^j{\cal M})|^2 F,F \rangle
= \sum_{j=-\infty}^{\infty} \|g(2^j{\cal M}) F\|^2_2.
\end{equation}
But since $g$ is supported in $[1/2,2]$, we may, for each $j$, choose $l(j)$ with $g(2^j\nu_l) = 0$
whenever $l > l(j)$, so that $g(2^j{\cal M}): {\cal H} \rightarrow {\cal P}_{l(j)}{\cal H}$.
As is well known, for any $l$, the product of two elements of ${\cal P}_l{\cal H}$ is in
${\cal P}_{2l}{\cal H}$.  Accordingly $|g(2^j{\cal M}) F|^2 \in {\cal P}_{2l(j)}{\cal H}$, and so
$\|g(2^j{\cal M}) F\|^2_2 = \int |g(2^j{\cal M}) F|^2$ may be evaluated exactly by cubature
(instead of our having to approximate it by a Riemann sum, as in (\ref{kerfajap})).  We find that
\begin{equation}\notag
\|F\|_2^2 =  \sum_{j=-\infty}^{\infty} \|g(2^j{\cal M}) F\|^2_2 =
\sum_{j=-\infty}^{\infty} \sum_i \lambda_{2l(j),i} |g(2^j{\cal M}) F|^2(x_{2l(j),i})
= \sum_{j=-\infty}^{\infty} \sum_i |\langle F,\phi_{j,i} \rangle|^2,
\end{equation}
where
\begin{equation}
\label{vphkq}
\phi_{j,i}(y) = \sqrt{\lambda_{2l(j),i}}\: \overline{K}_{2^j}(x_{j,i},y)
\end{equation}
(analogously to (\ref{phjkdf})). The $\{\phi_{j,i}\}$ are therefore a normalized tight frame,
not just a nearly tight frame.  Moreover, the constraints on the supports of the dyadic dilates of $g$,
easily imply that frame elements at non-adjacent scales are orthogonal.\\

We turn to apparent disadvantages of needlets.  In the study of CMB, they are not implemented directly on the
sphere, for several reasons which we shall explain in a moment, including lack of usable formulas.
In order to evaluate the inner product $\langle F,\phi_{j,i} \rangle$, one needs the spherical harmonic expansion
of $F$; then one uses $g(t{\cal M})F = g(t\sum_l lP_l)F$ to evaluate the inner products.
(See the bottom of page 9 of \cite{guil}.)  In CMB there is a large region of missing data on the sphere,
called the ``sky cut'', arisng from interference from the brightness of the Milky Way.  Thus,
finding spherical harmonic coefficients, which depend essentially on the global behavior of $F$,
is problematic.  In contrast, when Mexican needlets are used, one can effectively evaluate the
inner products if $x_{j,i}$ is somewhat away from the sky cut (on the scale $j$), because of the Gaussian
decay of the $\varphi_{j,i}$ at each scale and the lack of oscillation of Mexican needlets.
(Here we are of course assuming that the approximation
(\ref{htapp}) is rigorously justified.)\\

As to whether effective formulas could someday be found for needlets that could be used
directly on the sphere, we are pessimistic, for the following (admittedly circumstantial) reasons.
If one takes the inverse Fourier transform of a cutoff
function, which resembles a characteristic function, then one would expect an answer
which looks like the inverse Fourier transform of a characteristic function, i.e., something that looks like
the oscillatory function
$\sin(x)/x$ (of course it must be in the Schwartz space).  One would expect something similar to happen on
the sphere, and in fact, if one takes $t = .1$ and $g(s)$ to be the cutoff function
$exp(-1/[9/16-(s-5/4)^2])$ for $s \in [1/2,2]$, Maple says that $4\pi h_t(\cos \theta)$ is as in Figure 2, which should
be contrasted with the better-behaved function in Figure 1, where we took $t=.1$ and $f(s)=se^{-s}$.)

\begin{figure}
  \begin{center}
  \begin{tabular}{c}
  \includegraphics[scale=0.5]{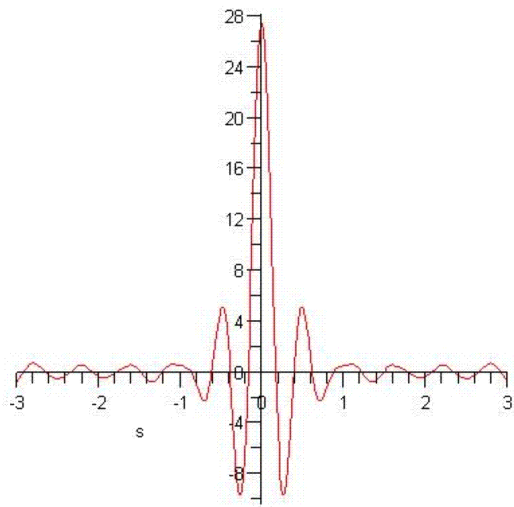}
  \end{tabular}
 \caption{\label{fig 7.pdf} A wavelet on $S^2$ obtained from a cutoff function ($t=0.1$)}
  \end{center}
  \end{figure}

  Note also that, since a cutoff function is not real analytic, its inverse Fourier transform cannot decay
exponentially; one would expect something similar on the sphere.  Also note that
on the real line, if one wants a function $g$ such that both
$g$ and $\check{g}$ are very small outside compact sets, the uncertainty principle says that a Gaussian is the best choice; one
would expect a similar phenomenon on the sphere.\\

Errors of one sort or another being unavoidable, it would be worthwhile to utilize both needlets and Mexican
needlets in the analysis of CMB, and the results should be compared.  We also
suggest that a ``hybrid'' approach
be attempted, combining the ideas of this section with those of
\cite{narc1}, \cite{narc2}.  Instead of using a cutoff function, we let $f(s) = se^{-s}$,
let $K_t$ be the kernel
of $f(t^2\Delta)$, and again define $\phi_{j,i}$ by (\ref{vphkq})
(for suitable $l(j)$).  We would then expect
(from (\ref{htapp})) that we will be able to evaluate the $\langle F,\phi_{j,i} \rangle$
without first finding the spherical harmonic decomposition of $F$.   However, the
$\{\phi_{j,i}\}$ are then only a nearly tight frame, for the following two reasons:

\begin{itemize}
\item[a)]
  We have only (\ref{daub}) and (\ref{daubest}) instead of point \#1 above.
 \item[b)] $f$ does not have compact support, so that cubature formulas will not exactly hold.
 \end{itemize}

However, these issues should lead only to very small errors.  For point a), we recall that
$B_a/A_a \rightarrow 1$ nearly quadratically in (\ref{daubest}).  We simply need to
use dilations by $a^j$ instead of $2^j$
for $j$ sufficiently close to $1$.  (As we have said, if $a = 2^{1/3}$, then
$B_a/A_a = 1.0000$ to four significant digits.)   For point b), we note
that $f$ has exponential decay at infinity, so that it basically has compact support for all practical
purposes. \\

Let us outline an argument to explain why the $\{\phi_{j,i}\}$, as in (\ref{vphkq}), associated to
$f(s) = se^{-s}$ are extremely close to being a tight frame (for suitable $l(j)$).
It would not
be at all difficult to make this argument rigorous.  (In the terminology above,
we shall restrict attention to ${\cal M} = \sqrt{\Delta}$.)  We will pay a price
in that $l(j)$ will need to increase with $j$, ever so slightly faster than in
\cite{narc1}, \cite{narc2}.

Before we begin this argument, it is best to give some elementary estimates that we shall need.
First, in (\ref{cald}) and (\ref{daubest}), with $f(t) = te^{-t}$, note that $c=\int_0^{\infty} te^{-2t}dt
= 1/4$.  Also $te^{-2t}$ is decreasing for $t > 1/2$, and if $M > 0$,
\begin{equation}\notag
\int_{M}^{\infty} te^{-2t} dt = e^{-2M}\left(\frac{M}{2}+\frac{1}{4}\right).
\end{equation}
On the other hand, if $C, b > 0$ and $g$ is a continuous function such that
$|g(x)|^2/x$ is decreasing on $[C,\infty)$, then by the method of the integral test,
$$\sum_{\{j: a^{2(j-1)} > C/b\}} (|g(ba^{2j})|^2/a^{2j})(a^{2j}-a^{2(j-1)}) \leq \int_C^{\infty} (|g(t)|^2/t) dt.$$
Applying this to $g(s) = se^{-s}$ we find that
\begin{equation}
\label{intminf1}
\sum_{\{j : a^{2(j-1)} > M/b \}} |f(ba^{2j})|^2 \leq \frac{a^2}{a^2-1}e^{-2M}\left(\frac{M}{2}+\frac{1}{4}\right),
\end{equation}
if $M > 1/2$.\\

Seoondly, for $j \in \ZZ$, let $j_- = \max(-j,0)$.  Say $a > 1$, $N, l, r \geq 1$; we shall need
a very crude estimate for the solution $m=m(N,r,l)$ of $a^{2m}l(l+n-1) = N+rm_-$. (Note that the left side increases as
$m$ increases, with range $(0,\infty)$, while the right side is positive and nonincreasing as $m$ increases; so there is
one and only one solution.)  Let $p =
\frac{1}{2}\log_a(N/[l(l+n-1)])$.  If $l(l+n-1) < N$, then $m$ must be positive, and hence equal to $p$.
If $l(l+n-1) \geq N$, then $m$ must be nonpositive.
In that situation, $a^{2p}l(l+n-1) = N \leq N+rp_-$,
so $m \geq p$.  On the other hand, let $q= \frac{1}{2}\log_a(rN/l) = \frac{1}{2}[\log_a(rN)-\log_a l]$,
so that $q_- \leq \frac{1}{2}\log_a l$.  Since $xy \geq x+y$ if $x,y \geq 2$,
we have $$a^{2q}[l(l+n-1)] = rN(l+n-1) \geq N + r(l+n-1) \geq N + r\log_a l \geq N+rq_-;$$ so
$m \leq q$.  Altogether, we always have
\begin{equation}
\label{mnrlest}
\frac{1}{2}\log_a(N/[l(l+n-1)]) \leq m(N,r,l) \leq \frac{1}{2}\log_a(rN/l).
\end{equation}

We now present our argument.
Since, in (\ref{daubest}), $c = 1/4$, we may choose $a > 1$
sufficiently close to $1$ that $A_a > \frac{1-\epsilon_1}{8\log a}$,
$B_a < \frac{1+\epsilon_1}{8\log a}$, where $\epsilon_1 << 1$.  Say again that $(I-P_0)F=F$.  In place of
(\ref{nar1}) we have that for some $\eta_1$ with $|\eta_1| < \epsilon_1$,
\begin{equation}\notag
\left(\frac{1+\eta_1}{8 \log a}\|F\|_2^2\right)^{1/2} =
\left(\langle \sum_{j=-\infty}^{\infty} |f(a^{2j}\Delta)|^2 F,F \rangle\right)^{1/2}
= \left(\sum_{j=-\infty}^{\infty} \|f(a^{2j}\Delta) F\|^2_2\right)^{1/2}.
\end{equation}
We repeat that $a$ has now been fixed; we next need to
choose $N > 1/2$ large (how large will be explained in the following argument).  For $j \in \ZZ$, let $j_- = \max(-j,0)$,
$r=\max(1,2(n+2)\ln a)$, and set $f_{1,j} = \chi_{[0,Na^2+ra^2(j-1)_-]}f$,
$f_{2,j} = f - f_{1,j}$.
Since an $\ell^2$ norm is a norm, we can write
\[ \left(\sum_{j=-\infty}^{\infty} \|f(a^{2j}\Delta) F\|^2_2\right)^{1/2}
= \left(\sum_{j=-\infty}^{\infty} \|f_{1,j}(a^{2j}\Delta) F\|^2_2\right)^{1/2} + E_1, \]
where $E_1 \leq \left(\sum_{j=-\infty}^{\infty} \|f_{2,j}(a^{2j}\Delta) F\|^2_2\right)^{1/2} := {\cal A}(F)$, say.
Since $f_{1,j}(s)$ vanishes for $s > Na^2+ra^2(j-1)_-$,\\
$f_{1,j}(a^{2j}l(l+n-1))=0$ if $l > \sqrt{N+r(j-1)_-}/a^{j-1}$; let
$l(j)$ be the least integer greater than $\sqrt{N+r(j-1)_-}/a^{j-1}$.  Thus
\[ \left(\sum_{j=-\infty}^{\infty} \|f_{1,j}(a^{2j}{\Delta}) F\|^2_2\right)^{1/2} =
\left(\sum_{j=-\infty}^{\infty} \sum_i \lambda_{2l(j),i} |f_{1,j}(a^{2j}{\Delta}) F|^2(x_{2l(j),i})\right)^{1/2}\]

Letting $K_t(x,y)$ be the kernel of $f(t^2\Delta)$, we see that this equals

\[ \left(\sum_{j=-\infty}^{\infty} \sum_i \lambda_{2l(j),i} |f(a^{2j}{\Delta}) F|^2(x_{2l(j),i})\right)^{1/2} + E_2
= \left(\sum_{j=-\infty}^{\infty} \sum_i |\langle F,\phi_{j,i} \rangle|^2\right)^{1/2} + E_2, \]
where $\phi_{j,i}(y) = \sqrt{\lambda_{2l(j),i}}\: \overline{K}_{a^j}(x_{j,i},y)$, and
$E_2 \leq [\sum_{j=-\infty}^{\infty} \sum_i \lambda_{2l(j),i} |f_{2,j}(a^{2j}{\Delta}) F|^2(x_{2l(j),i})]^{1/2}
:= {\cal B}(F)$, say.
If we can show that (for $N$ sufficiently large) that
${\cal A}(F), {\cal B}(F) < \epsilon_2 \|F\|$, where $\epsilon_2 << 1$,
then by combining the above facts we find that
\[ \frac{1+\eta_2}{8 \log a}\|F\|_2^2 = \sum_{j=-\infty}^{\infty} \sum_i |\langle F,\phi_{j,i} \rangle|^2\]
for some $\eta_2 << 1$, as we wanted.  But
\[ {\cal A}(F)^2 = \langle \sum_{j=-\infty}^{\infty} |f_{2,j}(a^{2j}\Delta)|^2 F,F \rangle
\leq \epsilon_3 \|F\|^2\]
if
$$\epsilon_3 = \max_{s > 0}\sum_{j=-\infty}^{\infty} |f_{2,j}(a^{2j}s)|^2 \leq
\max_{s > 0} \sum_{a^{2j}>Na^2/s} |f(a^{2j}s)|^2 < Ce^{-N}N << 1$$
if $N$ is sufficiently large.  (We have used (\ref{intminf1}) and recalled that $N > 1/2$.)
 As for ${\cal B}(F)$, denote the kernel of
$f_{2,j}(t^2{\Delta})$ by $J_t(x,y) = J_t^x(y)$.  Then for any $x \in S^n$,
$|f_{2,j}(a^{2j}{\Delta}) F|^2(x)] \leq \|J_{a^j}^x\|_2^2\|F\|_2^2$.  For any orthogonal transformation $T$,
$J_t(Tx,Ty) = J_t(x,y)$ for all $x,y$; accordingly $\|J_{a^j}^x\|_2^2$ is {\em independent} of $x$, so it
equals $$\omega_n^{-1}\int_{S^n}\|J_{a^j}^{x}\|_2^2 dS(x) =
\omega_n^{-1}\|f_{2,j}(a^{2j}{\Delta})\|_2^2,$$
 where the last $\|\:\|_2$ denotes
Hilbert-Schmidt norm.  On the other hand, for any $j$, $\sum_i \lambda_{2l(j),i} =
\sum_i \lambda_{2l(j),i}1 = \omega_n$, since $1 \in {\cal P}_{l(j)}{\cal H}$.  Putting these facts
together we see that ${\cal B}(F)^2 \leq \epsilon_4 \|F\|^2$, where

\begin{align}\notag
 \epsilon_4  = \sum_{j=-\infty}^{\infty} \|f_{2,j}(a^{2j}{\Delta})\|_2^2
 &  =
\sum_{j=-\infty}^{\infty} \sum_{l=1}^{\infty}|f_{2,j}(a^{2j}l(l+n-1))|^2 \dim{\cal H}_l \\\notag&\leq
C\sum_{l=1}^{\infty} l^n \sum_{j=-\infty}^{\infty} |f_{2,j}(a^{2j}l(l+n-1))|^2\\\notag
&
\leq C\sum_{l=1}^{\infty} l^n \sum_{a^{2j} > (Na^2+ra^2(j-1)_-)/[l(l+n-1)]} |f(a^{2j}l(l+n-1))|^2.\notag
\end{align}
In the inner summation, $a^{2(j-1)}l(l+n-1) >  N + r(j-1)_-$, so $j-1 > m(N,r,l)$ (notation as in
(\ref{mnrlest})).  Accordingly, in the inner summation, if we set $m = m(N,r,l)$, then
\begin{equation}
\label{a2jgeq}
a^{2(j-1)}l(l+n-1) > a^{2m}l(l+n-1) = N + rm_-.
\end{equation}
Note also that, by (\ref{mnrlest}), $N+rm_- \geq N+rc(N,l)$,
where we set

\begin{align}\notag
c(N,l) =\begin{cases}
 \frac{1}{2}\left(\log_a l - \log_a rN\right) &\text{if }
   l(l+n-1) \geq N ,\\
\notag
 0 &  \text{otherwise}.
\end{cases}
\end{align}
  By
(\ref{intminf1}) and (\ref{a2jgeq}), then,  $\epsilon_4 \leq C\sum_{l=1}^{\infty} l^n e^{-N-rc(N,l)}(N+r[m(N,r,l)]_-)$.
Since $exp(-\frac{r}{2}\log_a l) \leq a^{-(n+2)\log_a l} = l^{-n-2}$, we easily find (by (\ref{mnrlest})) that
$$\epsilon_4 \leq Ce^{-N}N^{n+3} << 1,$$
as desired.

\end{document}